\documentclass[a4paper]{amsart}
\usepackage{amsmath,amsthm,amssymb,latexsym,epic,bbm,comment}
\usepackage{graphicx,enumerate,stmaryrd}
\usepackage[all,2cell]{xy}
\xyoption{2cell}

\newtheorem{theorem}{Theorem}
\newtheorem{lemma}[theorem]{Lemma}
\newtheorem{corollary}[theorem]{Corollary}
\newtheorem{proposition}[theorem]{Proposition}

\usepackage[all]{xy}
\usepackage[active]{srcltx}
\usepackage[parfill]{parskip}
\usepackage{enumerate}

\font\sc=rsfs10
\newcommand{\cC}{\sc\mbox{C}\hspace{1.0pt}}

\font\scc=rsfs7
\newcommand{\ccC}{\scc\mbox{C}\hspace{1.0pt}}

\begin{document}
\title[Simple transitive 2-representations for projective functors]
{Simple transitive $2$-representations for two non-fiat $2$-categories of projective functors}

\author{Volodymyr Mazorchuk and Xiaoting Zhang}

\begin{abstract}
We show that any simple transitive $2$-representation of the $2$-ca\-te\-go\-ry of projective
endofunctors for the quiver algebra  of $\Bbbk(\xymatrix{\bullet\ar[r]&\bullet})$ and
for the quiver algebra of $\Bbbk(\xymatrix{\bullet\ar[r]\ar@/^/@{.}[rr]&\bullet\ar[r]&\bullet})$
is equivalent to a cell $2$-rep\-re\-sen\-ta\-ti\-on.
\end{abstract}

\maketitle

\section{Introduction and description of the results}\label{s0}

Classification problems are interesting and important problems in the classical representation theory.
For example, classifications of various classes of simple or indecomposable modules over
different classes of algebras played significant role in both development and applications
of modern representation theory.

Higher representation theory is a recent direction of representation theory that
takes its origins from the papers \cite{BFK,CR,Ro1,Ro2}. Of particular interest in
higher representation theory is the study of so-called finitary $2$-categories as
the latter are natural $2$-analogues of finite dimensional algebras.
Initial abstract study of finitary $2$-categories and the corresponding $2$-representation
theory was done in  \cite{MM1,MM2,MM3,MM4,MM5,MM6,Xa}.

As an outcome of this study, one interesting and important class of $2$-representations, called
{\em simple transitive $2$-representations}, was defined in \cite{MM5}. These $2$-rep\-re\-sen\-ta\-tions are
natural $2$-analogues of usual simple modules over algebras. Therefore the problem of classification of
simple transitive $2$-representations is natural and interesting. In several cases, it turns out that
simple transitive $2$-representations can be classified, see for example various results in
\cite{MM5,MM6,GM1,Zh2,Zi}. We also refer the reader to \cite{GM2,KM,Ma,Zh1} to related questions
and applications. In particular, in \cite{KM}, classification of simple transitive $2$-representations
for the $2$-category of Soergel bimodules over the coinvariant algebra of the symmetric group was
crucially used  for classification of projective functors in parabolic category $\mathcal{O}$ for $\mathfrak{sl}_n$.

The most basic example of a $2$-category is the $2$-category $\cC_A$ of projective functors
for a finite-dimensional algebra $A$ over an algebraically closed field $\Bbbk$, defined in
\cite[Subsection~7.3]{MM1}. In \cite{MM3,MM6}, it is shown that categories of the form $\cC_A$
essentially exhausts a natural class of ``simple'' finitary $2$-categories possessing weak
involutions. For such $2$-categories, it was shown in \cite{MM5,MM6} that simple transitive
$2$-representations are exactly the cell $2$-representations, defined in  \cite{MM1}.
Existence of a weak involution on a $2$-category restricts the classification result to the
case when $A$ is a self-injective algebra.

The aim of the present paper to classify simple transitive $2$-representations of $\cC_A$
for the smallest possible non-self-injective algebra, namely the path algebra $A$ of the quiver
$1\longrightarrow 2$, over an algebraically closed field $\Bbbk$. It turns out that our approach
also extends, with a significantly increased amount of technical work, to the quiver algebra  of
$\Bbbk(\xymatrix{\bullet\ar[r]\ar@/^/@{.}[rr]&\bullet\ar[r]&\bullet})$, where, as usual, the
dotted arrow depicts the corresponding zero relation. Our main result is the following
theorem, we refer the reader to Sections~\ref{s0-1} for details on all definitions.

\begin{theorem}\label{mainresult}
For $A=\Bbbk(\bullet\to\bullet)$ or $A=\Bbbk(\xymatrix{\bullet\ar[r]\ar@/^/@{.}[rr]&\bullet\ar[r]&\bullet})$,
any simple transitive $2$-representation of the $2$-category $\cC_A$ is equivalent to a  cell $2$-representation.
\end{theorem}

Despite of the fact that the formulation of Theorem~\ref{mainresult} is rather similar to
the corresponding statement in the case when $A$ is self-injective, considered in \cite{MM5,MM6},
our approach to the proof is fairly different, since the general approach outlined in \cite{MM5,MM6}
does not apply. Our approach, rather, has many similarities with the approach in \cite{Zi}
and is mostly based on a careful analysis of all possible cases.

In Section~\ref{s0-1} we collect all necessary preliminaries for $2$-representation theory
of the $2$-category $\cC_A$. In Section~\ref{s51} we prove some general results about
$2$-representations of $\cC_A$ under the additional assumption that the algebra $A$ has
a non-zero projective injective module.

In Section~\ref{s2}, \ref{s1} and \ref{s7}, we collect the proof of Theorem~\ref{mainresult}
in the case $A=\Bbbk(\bullet\to\bullet)$. In more details, Section~\ref{s2} contains preliminaries
on $\cC_A$ for $A=\Bbbk(\bullet\to\bullet)$. Section~\ref{s1} contains combinatorial
results on certain integer matrices which allow us to specify three essentially different
cases which we have to deal with. In Sections~\ref{s7} we prove Theorem~\ref{mainresult}
for $A=\Bbbk(\bullet\to\bullet)$.

In Sections~\ref{s21}, \ref{s11}, \ref{s12n} and \ref{s19}, we collect the proof of Theorem~\ref{mainresult}
in the second case of the algebra $A=\Bbbk(\xymatrix{\bullet\ar[r]\ar@/^/@{.}[rr]&\bullet\ar[r]&\bullet})$.
In more details, in Section~\ref{s21} one finds preliminaries on $\cC_A$. Sections~\ref{s11} and \ref{s12n}
are devoted to finding out which integer matrix captures the combinatorics of
a faithful simple transitive $2$-representation of $\cC_A$. Finally,
Sections~\ref{s19} completes the proof of  Theorem~\ref{mainresult} for
the algebra $\Bbbk(\xymatrix{\bullet\ar[r]\ar@/^/@{.}[rr]&\bullet\ar[r]&\bullet})$.
\vspace{5mm}

{\bf Acknowledgment.} The first author is partially supported by the Swedish Research Council.
We thank Vanessa Miemietz for stimulating discussions.

\section{$2$-category $\cC_A$ and its $2$-representations}\label{s0-1}

\subsection{Notation and conventions}\label{s0-1.1}

Throughout the paper we work over an algebraically closed field $\Bbbk$ and abbreviate $\otimes_{\Bbbk}$
by $\otimes$. Unless explicitly stated otherwise, by a module, we mean a {\em left} module.
We compose maps from right to left. For a $1$-morphism $\mathrm{F}$, we denote by
$\mathrm{id}_{\mathrm{F}}$ the identity $2$-morphism for $\mathrm{F}$.

\subsection{$2$-category $\cC_A$}\label{s0-1.2}

We refer the reader to \cite{Le,Mc,Ma} for generalities on $2$-categories. A $2$-category is a
category which is enriched over the monoidal category $\mathbf{Cat}$ of small categories.

Let $A$ be a connected, basic, finite dimensional $\Bbbk$-algebra and $\mathcal{C}$ a small
category equivalent to  $A$-mod. Consider the $2$-category $\cC_A$
(which depends on $\mathcal{C}$) defined as follows:
\begin{itemize}
\item $\cC_A$ has one object $\mathtt{i}$ which we identify with $\mathcal{C}$;
\item $1$-morphisms in $\cC_A$ are endofunctors of $\mathcal{C}$ isomorphic to functors given by
tensoring with $A$-$A$-bimodules from the additive closure of both ${}_AA_A$ and ${}_AA\otimes A_A$;
\item $2$-morphisms in $\cC_A$ are natural transformations of functors.
\end{itemize}
The $2$-category $\cC_A$ is {\em finitary} in the sense of \cite[Subsection~2.2]{MM1}.

\subsection{$2$-representations $\cC_A$}\label{s0-1.3}

We consider the $2$-category $\cC_A$-afmod of all {\em finitary $2$-representations} of $\cC_A$.
In this $2$-category,
\begin{itemize}
\item objects are strict additive functorial actions of $\cC_A$ on additive, idempotent split,
Krull-Schmidt, $\Bbbk$-linear categories with finitely many isomorphism classes of indecomposable
objects and finite dimensional morphism spaces;
\item $1$-morphisms are $2$-natural transformations;
\item $2$-morphisms are modifications.
\end{itemize}
We refer the reader to \cite{MM3} for details.
Two $2$-representations are called {\em equivalent} provided that there is a
$2$-natural transformation between them whose restriction to each object  is an equivalence
of categories.

We also consider the $2$-category $\cC_A$-mod defined similarly using functorial action on
categories equivalent to module categories of finite dimensional $\Bbbk$-algebras.
There is the diagrammatically defined abelianization $2$-functor
\begin{displaymath}
\overline{\hspace{1mm}\cdot\hspace{1mm}}:\cC_A\text{-}\mathrm{afmod}\to \cC_A\text{-}\mathrm{mod},
\end{displaymath}
see \cite[Subsection~4.2]{MM2} for details.

A finitary  $2$-representation  of $\cC_A$ is called {\em transitive} provided that,
for any indecomposable objects $X$ and $Y$ in $\mathbf{M}(\mathtt{i})$, there is a $1$-morphism
$\mathrm{F}$ in $\cC_A$ such that $Y$ is isomorphic to a direct summand of $\mathbf{M}(\mathrm{F})\, X$.

A transitive $2$-representation $\mathbf{M}$ is called {\em simple} provided that
$\mathbf{M}(\mathtt{i})$ does not have non-zero proper $\cC_A$-invariant ideals.

For simplicity, we will often use the module notation $\mathrm{F}\, X$ instead of the
representation notation $\mathbf{M}(\mathrm{F})\, X$.

\subsection{Cells in $\cC_A$}\label{s0-1.4}

Let $1=e_1+e_2+\dots+e_n$ be a primitive decomposition of $1\in A$. Up to isomorphism, indecomposable
$1$-morphisms in $\cC_A$ are given by tensoring with ${}_AA_A$ or with ${}_AAe_i\otimes e_jA_A$,
where $i,j=1,2,\dots,n$. We fix a representative $\mathrm{F}_0$ in the isomorphism class
of $1$-morphisms which correspond to tensoring with ${}_AA_A$. For $i,j=1,2,\dots,n$,
we fix a representative  $\mathrm{F}_{ij}$ in the isomorphism class
of $1$-morphisms which correspond to tensoring with ${}_AAe_i\otimes e_jA_A$. The set of isomorphism classes
of indecomposable $1$-morphisms in $\cC_A$ has the natural structure of a multisemigroup, see
\cite[Section~3]{MM2} and \cite{KuMa}. Combinatorics of this structure is encoded into so-called
{\em left, right} and {\em  two-sided cells}. For $\cC_A$, the two sided-cells are
\begin{displaymath}
\mathcal{J}_0:=\{\mathrm{F}_0\} \quad\text{ and }\quad \mathcal{J}:=\{\mathrm{F}_{ij}\,:\,i,j=1,2,\dots,n\}.
\end{displaymath}
The two-sided cell $\{\mathrm{F}_0\}$ is a left and a right cell as well. Other left cells are
\begin{displaymath}
\{\mathrm{F}_{ij}\,:\,i=1,2,\dots,n\},\quad j=1,2,\dots,n.
\end{displaymath}
Other right cells are
\begin{displaymath}
\{\mathrm{F}_{ij}\,:\,j=1,2,\dots,n\},\quad i=1,2,\dots,n.
\end{displaymath}
As usual, we have
\begin{equation}\label{eqnn1}
\mathrm{F}_{ij}\circ\mathrm{F}_{st} =\mathrm{F}_{it}^{\oplus \dim(e_jAe_s)}.
\end{equation}
We set
\begin{displaymath}
\mathrm{F}:=\bigoplus_{i,j=1}^n \mathrm{F}_{ij}
\end{displaymath}
and note that
\begin{equation}\label{eqnn0}
\mathrm{F}\circ\mathrm{F}\cong \mathrm{F}^{\oplus\dim(A)}
\end{equation}
All $1$-morphisms in the additive closure of $\mathrm{F}$ are called
{\em projective endofunctors} of $\mathcal{C}$. Similarly for $A$-mod.

As usual, we will say that a pair $(\mathrm{F}_{ij},\mathrm{F}_{st})$ of
$1$-morphisms is a pair of {\em adjoint} $1$-morphisms provided that there exist $2$-morphisms
\begin{displaymath}
\alpha:\mathrm{F}_{ij}\mathrm{F}_{st}\to \mathrm{F}_{0}\quad\text{ and }\quad
\beta:\mathrm{F}_{0}\to \mathrm{F}_{st}\mathrm{F}_{ij}
\end{displaymath}
such that
\begin{displaymath}
(\alpha\circ_0 \mathrm{id}_{\mathrm{F}_{ij}})\circ_1
(\mathrm{id}_{\mathrm{F}_{ij}}\circ_0\beta)=\mathrm{id}_{\mathrm{F}_{ij}}
\quad\text{ and }\quad
(\mathrm{id}_{\mathrm{F}_{st}}\circ_0\alpha)\circ_1
(\beta\circ_0\mathrm{id}_{\mathrm{F}_{st}})=\mathrm{id}_{\mathrm{F}_{st}}.
\end{displaymath}

The $2$-category $\cC_A$ is $\mathcal{J}$-simple in the sense that any non-zero two-sided $2$-ideal of
$\cC_A$ contains the identity $2$-morphisms for all $1$-morphisms given by projective endofunctors,
see \cite{MM2,Ag}.

\subsection{Cell $2$-representations}\label{s0-1.5}

The first example of a finitary $2$-representation of $\cC_A$ is the {\em principal}
$2$-representation $\mathbf{P}:=\cC_A(\mathtt{i},{}_-)$. This has a unique maximal
$\cC_A$-invariant ideal and the corresponding quotient is the cell
$2$-representation $\mathbf{C}_{\mathcal{L}}$, where $\mathcal{L}=\{\mathrm{F}_0\}$.

For any other left cell $\mathcal{L}$, the additive closure of elements in $\mathcal{L}$
gives a $2$-sub\-re\-pre\-sen\-ta\-ti\-on of $\mathbf{P}$. This $2$-sub\-rep\-re\-sen\-ta\-ti\-on again has a
unique maximal  $\cC_A$-invariant ideal and the corresponding quotient is the cell
$2$-representation $\mathbf{C}_{\mathcal{L}}$. This latter cell $2$-representation is
equivalent to the defining action of $\cC_A$ on the category $A$-proj of projective
objects in $A$-mod, see \cite{MM1} for details.

\subsection{Matrices in the Grothendieck group}\label{s0-1.6}

Let $\mathbf{M}$ be a finitary $2$-representation of $\cC_A$ and $X_1$, $X_2$,\dots, $X_k$ be a
fixed complete and irredundant
list of representatives of isomorphism classes of indecomposable objects in $\mathbf{M}(\mathtt{i})$.
For a $1$-morphism $\mathrm{G}$ in $\cC_A$, we denote by $[\mathrm{G}]$ the $k\times k$ matrix with
non-negative integer coefficients where, for $i,j=1,2,\dots,k$, the coefficient in the intersection of the
$i$-th row and the $j$-th column gives the number of indecomposable direct summands of
$\mathbf{M}(\mathrm{G})\, X_j$ which are isomorphic to $X_i$. Note that
$[\mathrm{G}\oplus \mathrm{H}]=[\mathrm{G}]+[\mathrm{H}]$ and
$[\mathrm{G}\circ \mathrm{H}]=[\mathrm{G}][\mathrm{H}]$.

\subsection{Action on simple transitive $2$-representations}\label{s0-1.7}

The following statement is proved in \cite[Lemma~12]{MM5}.

\begin{lemma}\label{lemmamm5}
Let $\mathbf{M}$ be a simple transitive $2$-representation of $\cC_A$. Then, for any non-zero object
$X\in\overline{\mathbf{M}}(\mathtt{i})$, the object
$\mathrm{F}\,X$ is projective in $\overline{\mathbf{M}}(\mathtt{i})$.
\end{lemma}

The following statement is proved in \cite[Lemma~13]{MM5}.

\begin{lemma}\label{lemmamm7}
Let $B$ be a finite dimensional $\Bbbk$-algebra and $\mathrm{G}$ an exact
endofunctor of $B$-mod. Assume that $\mathrm{G}$ sends each simple object of $B$-mod
to a projective object. Then $\mathrm{G}$ is a projective functor.
\end{lemma}

\section{Existence of a projective-injective module guarantees exactness of the action}\label{s51}

\subsection{Exactness of the action of some projective functors}\label{s51.1}

Let $\mathbf{M}$ be a simple transitive $2$-representation of $\cC_A$.
Consider its abelianization $\overline{\mathbf{M}}$. For $\overline{\mathbf{M}}(\mathtt{i})$,
let $L_1$, $L_2$,\dots, $L_k$ be a complete and irredundant list of representatives of
isomorphism classes of simple objects. For $i\in\{1,2,\dots,k\}$,  denote by $P_i$ the
indecomposable projective cover of $L_i$ and by $I_i$ the indecomposable injective envelope of $L_i$.

\begin{lemma}\label{lem71-n}
Let $Q$ be a finite dimensional $\Bbbk$-algebra and $\mathrm{K}$ a right exact endofunctor of
$Q$-mod. Then the following conditions are equivalent:
\begin{enumerate}[$($a$)$]
\item\label{lem71-n.1} The functor $\mathrm{K}$ sends projective objects to projective objects.
\item\label{lem71-n.2} The right adjoint $\mathrm{K}'$ of $\mathrm{K}$ is exact.
\end{enumerate}
\end{lemma}

\begin{proof}
By adjunction, for a projective generator $P\in Q$-mod, we have a natural isomorphism
\begin{equation}\label{eq71e-n}
\mathrm{Hom}_Q(\mathrm{K}P,{}_-)\cong  \mathrm{Hom}_Q(P,\mathrm{K}'{}_-).
\end{equation}
If $\mathrm{K}P$ is projective, the left hand side of \eqref{eq71e-n} is exact. Hence the right hand
side is also exact. As $P$ is a projective generator, the functor $\mathrm{Hom}_Q(P,{}_-)$ detects
any non-zero homology. This forces $\mathrm{K}'$ to be exact. Therefore
\eqref{lem71-n.1} implies \eqref{lem71-n.2}.

Conversely, assume that $\mathrm{K}'$ is exact. Then the right hand side of  \eqref{eq71e-n} is exact.
Hence the left hand side is exact. This means that $\mathrm{K}P$ is projective. Therefore
\eqref{lem71-n.2} implies \eqref{lem71-n.1}. The claim follows.
\end{proof}

\begin{lemma}\label{lem72-n}
Assume that there exist $s,t\in\{1,2,\dots,n\}$ such that the left $A$-modules $Ae_s$ and
$\mathrm{Hom}_{\Bbbk}(e_tA,\Bbbk)$ are isomorphic. Then, for any $i\in\{1,2,\dots,n\}$, the pair
$(\mathrm{F}_{it},\mathrm{F}_{si})$ is a pair of adjoint  $1$-morphisms.
\end{lemma}

\begin{proof}
The functor
$\mathrm{F}_{it}$ is given by tensoring with the $A$-$A$-bimodule $Ae_i\otimes e_tA$. The
right adjoint of this functor is thus the functor $\mathrm{Hom}_A(Ae_i\otimes e_tA,{}_-)$.
By the computation in \cite[Subsection~7.3]{MM1}, the exact
functor $\mathrm{Hom}_A(Ae_i\otimes e_tA,{}_-)$ is isomorphic to the functor of tensoring
with the $A$-$A$-bimodule
\begin{displaymath}
\mathrm{Hom}_{\Bbbk}(e_tA,\Bbbk)\otimes e_iA.
\end{displaymath}
The injective $A$-module  $I_t\cong \mathrm{Hom}_{\Bbbk}(e_tA,\Bbbk)$ is isomorphic to the projective
$A$-mo\-dule $Ae_s$, by assumption. Therefore $\mathrm{Hom}_{\Bbbk}(e_tA,\Bbbk)\otimes e_iA$ is isomorphic to
$Ae_s\otimes e_iA$. This means that $\mathrm{F}_{si}$ is isomorphic to the right adjoint of
$\mathrm{F}_{it}$. The claim follows.
\end{proof}

\begin{corollary}\label{cor73-n}
Assume that there exist $s,t\in\{1,2,\dots,n\}$ such that the left $A$-mo\-dules $Ae_s$ and
$\mathrm{Hom}_{\Bbbk}(e_tA,\Bbbk)$ are isomorphic. Then, for any $i\in\{1,2,\dots,n\}$ and any
$2$-representation $\mathbf{N}$ of $\cC_A$, the pair
$(\mathbf{N}(\mathrm{F}_{it}),\mathbf{N}(\mathrm{F}_{si}))$
is a pair of adjoint functors.
\end{corollary}

\begin{proof}
This  follows directly from Lemma~\ref{lem72-n} and definitions.
\end{proof}

\begin{corollary}\label{cor74-n}
Assume that there exist $s,t\in\{1,2,\dots,n\}$ such that the left $A$-mo\-dules $Ae_s$ and
$\mathrm{Hom}_{\Bbbk}(e_tA,\Bbbk)$ are isomorphic. Then, for any $i\in\{1,2,\dots,n\}$ and any
finitary $2$-representation $\mathbf{N}$ of $\cC_A$, the functor
$\overline{\mathbf{N}}(\mathrm{F}_{si})$ is exact.
\end{corollary}

\begin{proof}
This  follows from the definitions by combining Lemma~\ref{lem71-n} and Corollary~\ref{cor73-n}.
\end{proof}

\subsection{Auxiliary lemma}\label{s51.2}

\begin{lemma}\label{lem74-n}
Let $Q$ be a finite dimensional $\Bbbk$-algebra and $\mathrm{K}$, $\mathrm{H}$ and $\mathrm{G}$ be two
endofunctor of $Q$-mod. Assume that:
\begin{enumerate}[$($a$)$]
\item\label{lem74-n.1} $\mathrm{H}$ is a projective functor;
\item\label{lem74-n.2} $\mathrm{K}$ is right exact;
\item\label{lem74-n.3} $\mathrm{K}$ sends projective objects to projective objects;
\item\label{lem74-n.4} $\mathrm{K}\circ \mathrm{H}\cong \mathrm{G}$.
\end{enumerate}
Then $\mathrm{G}$ is a projective functor.
\end{lemma}

\begin{proof}
By assumption~\eqref{lem74-n.1}, the functor $\mathrm{H}$ is given by tensoring with the $Q$-$Q$-bimodule
$X\otimes Y$, for some projective left $Q$-module $X$ and some projective right $Q$-module $Y$.
By assumption~\eqref{lem74-n.2}, $\mathrm{K}$ is given by tensoring with some $Q$-$Q$-bimodule
$V$. Using assumption~\eqref{lem74-n.4}, the $Q$-$Q$-bimodule that
determines the functor $\mathrm{G}$ is given by
\begin{equation}\label{eq74nn}
V\otimes_Q \left(X\otimes Y\right)\cong \left(V\otimes_Q X\right)\otimes Y.
\end{equation}
By assumption~\eqref{lem74-n.3}, $V\otimes_Q X$ is a projective left $Q$-module. This implies
that \eqref{eq74nn} is a projective $Q$-$Q$-bimodule and hence $\mathrm{G}$ is a projective functor.
\end{proof}

\subsection{Exactness of the action}\label{s51.3}

\begin{proposition}\label{prop77-n}
Assume that there exist $s,t\in\{1,2,\dots,n\}$ such that the left $A$-mo\-du\-les $Ae_s$ and
$\mathrm{Hom}_{\Bbbk}(e_tA,\Bbbk)$ are isomorphic. Let $\mathbf{M}$ be a simple
transitive $2$-representation of $\cC_A$. Then the functor
$\overline{\mathbf{M}}(\mathrm{F})$ is exact.
\end{proposition}

\begin{proof}
Let $B$ be a finite dimensional algebra such that $\overline{\mathbf{M}}(\mathtt{i})$
is equivalent to $B$-mod.

For $i\in\{1,2,\dots,n\}$, consider the $1$-morphism $\mathrm{F}_{si}$. By
Corollary~\ref{cor74-n}, the functor $\overline{\mathbf{M}}(\mathrm{F}_{si})$ is exact.
By Lemma~\ref{lemmamm5}, $\overline{\mathbf{M}}(\mathrm{F}_{si})$ sends any object in
$\overline{\mathbf{M}}(\mathtt{i})$ to a projective object. Therefore, by
Lemma~\ref{lemmamm7}, $\overline{\mathbf{M}}(\mathrm{F}_{si})$ is a projective endofunctor
of $B$-mod.

Now, for any $j\in \{1,2,\dots,n\}$, we have
\begin{displaymath}
\mathrm{F}_{js}\circ \mathrm{F}_{si}\cong \mathrm{F}_{ji}^{\oplus k},
\end{displaymath}
where $k=\dim(e_sAe_s)>0$. Therefore $\overline{\mathbf{M}}(\mathrm{F}_{ji}^{\oplus k})$ is a projective functor
for $B$-mod by Lemma~\ref{lem74-n}. By additivity,
$\overline{\mathbf{M}}(\mathrm{F}_{ji})$ is a projective functor
for $B$-mod  as well. In particular, $\overline{\mathbf{M}}(\mathrm{F}_{ji})$ is exact.
The claim follows.
\end{proof}

\section{The algebra $\Bbbk(\bullet\to\bullet)$}\label{s2}

Let $\Bbbk$ be an algebraically closed field. Denote by $A$ the path algebra, over $\Bbbk$, of the quiver
$1\overset{\alpha}{\longrightarrow}2$. The algebra $A$ has basis $e_1$, $e_2$ and $\alpha$
and the multiplication table $(x,y)\mapsto x\cdot y$ is given by:
\begin{displaymath}
\begin{array}{c||c|c|c}
x\backslash y & e_1 & e_2 & \alpha\\
\hline\hline
e_1 & e_1 & 0 & 0\\
\hline
e_2 & 0 & e_2 & \alpha\\
\hline
\alpha & \alpha & 0 & 0\\
\end{array}
\end{displaymath}
Note that $e_1Ae_2=0$ as $A$ contains no paths from $2$ to $1$. Note also that the left $A$-modules
$Ae_1$ and $\mathrm{Hom}_{\Bbbk}(e_2A,\Bbbk)$ are isomorphic.

Let $\mathcal{C}$ be a small category equivalent to $A$-mod. Consider the corresponding finitary
$2$-category $\cC_A$.
Up to isomorphism, indecomposable $1$-morphisms in $\cC_A$ are $\mathrm{F}_0$ and
$\mathrm{F}_{ij}$, where  $i,j=1,2$.
Note that formula~\eqref{eqnn0} for $A$ reads $\mathrm{F}\circ\mathrm{F}=\mathrm{F}^{\oplus 3}$.
Using \eqref{eqnn1}, the table of compositions for the functors $\mathrm{F}_{ij}$ (up to isomorphism) is as follows:
\begin{equation}\label{eq1}
\begin{array}{c||c|c|c|c}
\circ & \mathrm{F}_{11} & \mathrm{F}_{12} & \mathrm{F}_{21}& \mathrm{F}_{22}\\
\hline\hline
\mathrm{F}_{11} & \mathrm{F}_{11} & \mathrm{F}_{12} & 0& 0\\
\hline
\mathrm{F}_{12} & \mathrm{F}_{11} & \mathrm{F}_{12} & \mathrm{F}_{11}& \mathrm{F}_{12}\\
\hline
\mathrm{F}_{21} & \mathrm{F}_{21} & \mathrm{F}_{22} & 0& 0\\
\hline
\mathrm{F}_{22} & \mathrm{F}_{21} & \mathrm{F}_{22} & \mathrm{F}_{21}& \mathrm{F}_{22}\\
\end{array}
\end{equation}
Set $\mathcal{J}_0:=\{\mathrm{F}_{0}\}$ and $\mathcal{J}:=\{\mathrm{F}_{ij}\,:\,i,j=1,2\}$.
Note that the $2$-category $\cC_A$ is not weakly fiat in the sense of \cite{MM2,MM6} as
the algebra $A$ is not self-injective.

As $\cC_A$ is $\mathcal{J}$-simple and $A$ has trivial center, the only proper non-zero
quotient of  $\cC_A$ contains just the identity $1$-morphism (up to isomorphism) and its scalar
endomorphisms (cf. \cite{MM3}). Therefore this quotient is fiat with strongly regular $\mathcal{J}$-classes and hence
it has a unique, up to equivalence, simple transitive $2$-representation, namely $\mathbf{C}_{\mathcal{L}_0}$,
where $\mathcal{L}_0=\mathcal{J}_0$, see \cite[Theorem~18]{MM5}. This means that, in order to prove
Theorem~\ref{mainresult} for $A$, it is enough to consider {\em faithful} $2$-representations of $\cC_A$.

From the formula
\begin{equation}\label{eq10n}
\mathrm{Hom}_{A\text{-}A}(Ae_i\otimes e_jA,Ae_s\otimes e_tA)\cong e_iAe_s\otimes e_tAe_j,
\end{equation}
for all $i,j,s,t\in\{1,2\}$, we get the following table of
$\mathrm{Hom}_{\ccC_{\hspace{-1mm}A}(\mathtt{i})}(X,Y)$ (up to isomorphism), where $X$ and $Y$ are
indecomposable $1$-morphisms as specified in the table:
\begin{equation}\label{eq101}
\begin{array}{c||c|c|c|c}
X\setminus Y & \mathrm{F}_{11} & \mathrm{F}_{12} & \mathrm{F}_{21}& \mathrm{F}_{22}\\
\hline\hline
\mathrm{F}_{11} & \Bbbk & \Bbbk & 0& 0\\
\hline
\mathrm{F}_{12} & 0 & \Bbbk & 0 & 0\\
\hline
\mathrm{F}_{21} & \Bbbk & \Bbbk & \Bbbk& \Bbbk\\
\hline
\mathrm{F}_{22} & 0 & \Bbbk & 0& \Bbbk\\
\end{array}
\end{equation}

\section{Integer matrices for $\Bbbk(\bullet\to\bullet)$}\label{s1}

\subsection{Integer matrices satisfying $M^2=3M$}\label{s1.1}

In this section we classify all square matrices $M$ with positive integer coefficients which satisfy $M^2=3M$.

\begin{proposition}\label{prop1}
Let $M$ be a $k\times k$ matrix, for some $k$, with positive integer coefficients, satisfying $M^2=3M$. Then
$M$ is one of the following matrices:
\begin{gather*}
M_1:=\left(3\right),\quad
M_2:=\left(\begin{array}{cc}2&1\\2&1\end{array}\right),\quad
M_3:=\left(\begin{array}{cc}2&2\\1&1\end{array}\right),\quad
M_4:=\left(\begin{array}{cc}1&1\\2&2\end{array}\right),\\
M_5:=\left(\begin{array}{cc}1&2\\1&2\end{array}\right),\quad
M_6:=\left(\begin{array}{ccc}1&1&1\\1&1&1\\1&1&1\end{array}\right).
\end{gather*}
\end{proposition}

\begin{proof}
Clearly, we have $M_i^2=3M_i$, for each $i=1,2,3,4,5,6$. So, we need to show that no other square matrix with
positive integer coefficients satisfies $M^2=3M$.

Let $M$ be a $k\times k$ matrix, for some $k$, with positive integer coefficients satisfying $M^2=3M$.
Then $M$ is diagonalizable (as $x^2-3x$ has no multiple roots) and the only possible eigenvalues
for $M$ are $0$ and $3$.  From the Perron-Frobenius theorem it follows
that the Perron-Frobenius eigenvalue $3$ must have multiplicity one. Therefore $M$ has rank one and trace three.
As all entries in $M$ are positive integers, we get $k\leq 3$.

If $k=1$, then, clearly,  $M=M_1$.

If $k=3$, then all diagonal entries in $M$ are $1$. As all $2\times 2$ minors in $M$ should have determinant
zero and positive integer entries, it follows that all entries in $M$ are $1$ and thus $M=M_6$.

If $k=2$, then the two diagonal entries in $M$ are $1$ and $2$. As the determinant of $M$ is zero, the
two remaining entries are also $1$ and $2$. Therefore $M=M_i$, for some $i\in\{2,3,4,5\}$.
\end{proof}

\subsection{The matrix $[\mathrm{F}]$ for a faithful simple transitive $2$-representation}\label{s1.2}

Let $\mathbf{M}$ be a finitary, simple, transitive and faithful $2$-representation of $\cC_A$.
Let $M:=[\mathrm{F}]$ be the matrix of $\mathbf{M}(\mathrm{F})$ and, for $i,j=1,2$, let
$M_{ij}:=[\mathrm{F}_{ij}]$ be the matrix of $\mathbf{M}(\mathrm{F}_{ij})$. Note that
$M=M_{11}+M_{12}+M_{21}+M_{22}$.

The symmetric group $S_k$ acts on $\mathrm{Mat}_{k\times k}(\mathbb{Z})$ by conjugation
with permutation matrices. This action corresponds to permutation of basis elements,
whenever the matrix on which we act represents an endomorphism of some free $\mathbb{Z}$-module.
We will call this action the {\em permutation action}.

\begin{proposition}\label{prop2}
In order to respect the multiplication rule \eqref{eq1}, up to the permutation action,
we have the following three possibilities:
\begin{enumerate}[$($a$)$]
\item\label{prop2.1}
$M=M_2$ and
\begin{gather*}
M_{11}=\left(\begin{array}{cc}1&0\\ 0&0\end{array}\right),\,\,
M_{12}=\left(\begin{array}{cc}1&1\\ 0&0\end{array}\right),\,\,
M_{21}=\left(\begin{array}{cc}0&0\\ 1&0\end{array}\right),\,\,
M_{22}=\left(\begin{array}{cc}0&0\\ 1&1\end{array}\right).
\end{gather*}
\item\label{prop2.2}
$M=M_3$ and
\begin{gather*}
M_{11}=\left(\begin{array}{cc}0&1\\ 0&1\end{array}\right),\,\,
M_{12}=\left(\begin{array}{cc}1&0\\ 1&0\end{array}\right),\,\,
M_{21}=\left(\begin{array}{cc}0&1\\ 0&0\end{array}\right),\,\,
M_{22}=\left(\begin{array}{cc}1&0\\ 0&0\end{array}\right).
\end{gather*}
\item\label{prop2.3}
$M=M_6$ and
\begin{gather*}
M_{11}=\left(\begin{array}{ccc}1&0&0\\ 1&0&0\\ 0&0&0\end{array}\right),\,\,
M_{12}=\left(\begin{array}{ccc}0&1&1\\ 0&1&1\\ 0&0&0\end{array}\right),\\
M_{21}=\left(\begin{array}{ccc}0&0&0\\ 0&0&0\\ 1&0&0\end{array}\right),\,\,
M_{22}=\left(\begin{array}{ccc}0&0&0\\ 0&0&0\\ 0&1&1\end{array}\right).
\end{gather*}
\end{enumerate}
\end{proposition}

\begin{proof}
As $\mathbf{M}$ is simple, transitive and faithful, we get that  $M$ has positive integer entries.
As $\mathrm{F}\circ\mathrm{F}=\mathrm{F}^{\oplus 3}$, we have $M=M_i$ for some $i\in\{1,2,3,4,5,6\}$,
by Proposition~\ref{prop1}. As $M$ is the sum of four non-zero matrices
(corresponding to all $\mathrm{F}_{ij}$) each of which has non-negative integer entries, we have
$M\neq M_1$. The case $M=M_4$ reduces to the case $M=M_3$ by swapping the basis elements.
The case $M=M_5$ reduces to the case $M=M_2$ by swapping the basis elements. It is easy to check
that the cases \eqref{prop2.1}, \eqref{prop2.2} and \eqref{prop2.3} listed in the formulation
satisfy \eqref{eq1}.

{\bf Assume $M=M_2$.} Note, from \eqref{eq1}, that $\mathrm{F}_{11}$, $\mathrm{F}_{12}$ and $\mathrm{F}_{22}$
are idempotent, while $\mathrm{F}_{21}$ is nilpotent. Therefore $M_{11}$, $M_{12}$, $M_{22}$ must have
non-zero diagonals, while the diagonal for $M_{21}$ should be zero. From $M_{11}M_{22}=0$ it follows that
$M_{11}$ and $M_{22}$ cannot have common diagonal entries. In any case, this means that
$M_{12}$ has the non-zero diagonal entry in the left upper corner. Let us first assume the following:
\begin{displaymath}
M_{11}=\left(\begin{array}{cc}0&*\\ *&1\end{array}\right),\,\,
M_{12}=\left(\begin{array}{cc}1&*\\ *&0\end{array}\right),\,\,
M_{21}=\left(\begin{array}{cc}0&*\\ *&0\end{array}\right),\,\,
M_{22}=\left(\begin{array}{cc}1&*\\ *&0\end{array}\right).
\end{displaymath}
From $M_{11}M_{21}=0$, we get:
\begin{displaymath}
M_{11}=\left(\begin{array}{cc}0&0\\ 0&1\end{array}\right),\,\,
M_{12}=\left(\begin{array}{cc}1&0\\ *&0\end{array}\right),\,\,
M_{21}=\left(\begin{array}{cc}0&1\\ 0&0\end{array}\right),\,\,
M_{22}=\left(\begin{array}{cc}1&0\\ *&0\end{array}\right).
\end{displaymath}
This, however, contradicts $M_{11}M_{12}=M_{12}$. Now assume
\begin{displaymath}
M_{11}=\left(\begin{array}{cc}1&*\\ *&0\end{array}\right),\,\,
M_{12}=\left(\begin{array}{cc}1&*\\ *&0\end{array}\right),\,\,
M_{21}=\left(\begin{array}{cc}0&*\\ *&0\end{array}\right),\,\,
M_{22}=\left(\begin{array}{cc}0&*\\ *&1\end{array}\right).
\end{displaymath}
From $M_{11}M_{21}=M_{11}M_{22}=0$, we get:
\begin{displaymath}
M_{11}=\left(\begin{array}{cc}1&0\\ *&0\end{array}\right),\,\,
M_{12}=\left(\begin{array}{cc}1&1\\ *&0\end{array}\right),\,\,
M_{21}=\left(\begin{array}{cc}0&0\\ *&0\end{array}\right),\,\,
M_{22}=\left(\begin{array}{cc}0&0\\ *&1\end{array}\right).
\end{displaymath}
From $M_{12}M_{22}=M_{12}$, we get:
\begin{displaymath}
M_{11}=\left(\begin{array}{cc}1&0\\ 0&0\end{array}\right),\,\,
M_{12}=\left(\begin{array}{cc}1&1\\ 0&0\end{array}\right),\,\,
M_{21}=\left(\begin{array}{cc}0&0\\ 1&0\end{array}\right),\,\,
M_{22}=\left(\begin{array}{cc}0&0\\ 1&1\end{array}\right).
\end{displaymath}

{\bf Assume $M=M_3$.} Note, from \eqref{eq1}, that $\mathrm{F}_{11}$, $\mathrm{F}_{12}$ and $\mathrm{F}_{22}$
are idempotent, while $\mathrm{F}_{21}$ is nilpotent. Therefore $M_{11}$, $M_{12}$, $M_{22}$ must have
non-zero diagonals, while the diagonal for $M_{21}$ should be zero. From $M_{11}M_{22}=0$ it follows that
$M_{11}$ and $M_{22}$ cannot have common diagonal entries. In any case this means that
$M_{12}$ has the non-zero diagonal entry in the left upper corner. Let us first assume the following:
\begin{displaymath}
M_{11}=\left(\begin{array}{cc}1&*\\ *&0\end{array}\right),\,\,
M_{12}=\left(\begin{array}{cc}1&*\\ *&0\end{array}\right),\,\,
M_{21}=\left(\begin{array}{cc}0&*\\ *&0\end{array}\right),\,\,
M_{22}=\left(\begin{array}{cc}0&*\\ *&1\end{array}\right).
\end{displaymath}
From $M_{11}M_{21}=0$, we get:
\begin{displaymath}
M_{11}=\left(\begin{array}{cc}1&0\\ 0&0\end{array}\right),\,\,
M_{12}=\left(\begin{array}{cc}1&*\\ 0&0\end{array}\right),\,\,
M_{21}=\left(\begin{array}{cc}0&0\\ 1&0\end{array}\right),\,\,
M_{22}=\left(\begin{array}{cc}0&*\\ 0&1\end{array}\right).
\end{displaymath}
This, however, contradicts $M_{21}M_{12}=M_{22}$. Now assume the following:
\begin{displaymath}
M_{11}=\left(\begin{array}{cc}0&*\\ *&1\end{array}\right),\,\,
M_{12}=\left(\begin{array}{cc}1&*\\ *&0\end{array}\right),\,\,
M_{21}=\left(\begin{array}{cc}0&*\\ *&0\end{array}\right),\,\,
M_{22}=\left(\begin{array}{cc}1&*\\ *&0\end{array}\right).
\end{displaymath}
From $M_{11}M_{21}=M_{11}M_{22}=0$, we get:
\begin{displaymath}
M_{11}=\left(\begin{array}{cc}0&*\\ 0&1\end{array}\right),\,\,
M_{12}=\left(\begin{array}{cc}1&*\\ 1&0\end{array}\right),\,\,
M_{21}=\left(\begin{array}{cc}0&*\\ 0&0\end{array}\right),\,\,
M_{22}=\left(\begin{array}{cc}1&*\\ 0&0\end{array}\right).
\end{displaymath}
From $M_{11}M_{12}=M_{12}$, we get:
\begin{displaymath}
M_{11}=\left(\begin{array}{cc}0&1\\ 0&1\end{array}\right),\,\,
M_{12}=\left(\begin{array}{cc}1&0\\ 1&0\end{array}\right),\,\,
M_{21}=\left(\begin{array}{cc}0&1\\ 0&0\end{array}\right),\,\,
M_{22}=\left(\begin{array}{cc}1&0\\ 0&0\end{array}\right).
\end{displaymath}

{\bf Assume $M=M_6$.} Note, from \eqref{eq1}, that $\mathrm{F}_{11}$, $\mathrm{F}_{12}$ and $\mathrm{F}_{22}$
are idempotent, while $\mathrm{F}_{21}$ is nilpotent. Therefore $M_{11}$, $M_{12}$, $M_{22}$ must have
non-zero diagonals, while the diagonal for $M_{21}$ should be zero. Therefore, up to permutation
of basis vectors, we may assume that
\begin{gather*}
M_{11}=\left(\begin{array}{ccc}1&*&*\\ *&0&*\\ *&*&0\end{array}\right),\,\,
M_{12}=\left(\begin{array}{ccc}0&*&*\\ *&1&*\\ *&*&0\end{array}\right),\\
M_{21}=\left(\begin{array}{ccc}0&*&*\\ *&0&*\\ *&*&0\end{array}\right),\,\,
M_{22}=\left(\begin{array}{ccc}0&*&*\\ *&0&*\\ *&*&1\end{array}\right).
\end{gather*}
From $M_{11}M_{21}=M_{11}M_{22}=0$ we thus get that the last column of $M_{11}$ must be zero
and the first row of both $M_{21}$ and $M_{22}$ must be zero. Since the $M_{ij}$'s add up to $M$,
the rightmost element in the first row of $M_{12}$ must be $1$:
\begin{gather*}
M_{11}=\left(\begin{array}{ccc}1&*&0\\ *&0&0\\ *&*&0\end{array}\right),\,\,
M_{12}=\left(\begin{array}{ccc}0&*&1\\ *&1&*\\ *&*&0\end{array}\right),\\
M_{21}=\left(\begin{array}{ccc}0&0&0\\ *&0&*\\ *&*&0\end{array}\right),\,\,
M_{22}=\left(\begin{array}{ccc}0&0&0\\ *&0&*\\ *&*&1\end{array}\right).
\end{gather*}
From $M_{11}M_{12}=M_{12}$ it follows that the second row of $M_{11}$ cannot be zero, which yields:
\begin{gather*}
M_{11}=\left(\begin{array}{ccc}1&*&0\\ 1&0&0\\ *&*&0\end{array}\right),\,\,
M_{12}=\left(\begin{array}{ccc}0&*&1\\ 0&1&*\\ *&*&0\end{array}\right),\\
M_{21}=\left(\begin{array}{ccc}0&0&0\\ 0&0&*\\ *&*&0\end{array}\right),\,\,
M_{22}=\left(\begin{array}{ccc}0&0&0\\ 0&0&*\\ *&*&1\end{array}\right).
\end{gather*}
Now $M_{11}M_{12}=M_{12}$ implies that the first and the second rows of $M_{12}$ should coincide,
moreover, the first element in the third row in $M_{12}$ should be zero and also
the third row in $M_{11}$ and thus also in $M_{12}$ must be zero:
\begin{gather*}
M_{11}=\left(\begin{array}{ccc}1&0&0\\ 1&0&0\\ 0&0&0\end{array}\right),\,\,
M_{12}=\left(\begin{array}{ccc}0&1&1\\ 0&1&1\\ 0&0&0\end{array}\right),\\
M_{21}=\left(\begin{array}{ccc}0&0&0\\ 0&0&0\\ *&*&0\end{array}\right),\,\,
M_{22}=\left(\begin{array}{ccc}0&0&0\\ 0&0&0\\ *&*&1\end{array}\right).
\end{gather*}
Now, $M_{21}M_{11}=M_{21}$ gives:
\begin{gather*}
M_{11}=\left(\begin{array}{ccc}1&0&0\\ 1&0&0\\ 0&0&0\end{array}\right),\,\,
M_{12}=\left(\begin{array}{ccc}0&1&1\\ 0&1&1\\ 0&0&0\end{array}\right),\\
M_{21}=\left(\begin{array}{ccc}0&0&0\\ 0&0&0\\ *&0&0\end{array}\right),\,\,
M_{22}=\left(\begin{array}{ccc}0&0&0\\ 0&0&0\\ *&1&1\end{array}\right).
\end{gather*}
Finally, $M_{12}M_{21}=M_{11}$ gives:
\begin{gather*}
M_{11}=\left(\begin{array}{ccc}1&0&0\\ 1&0&0\\ 0&0&0\end{array}\right),\,\,
M_{12}=\left(\begin{array}{ccc}0&1&1\\ 0&1&1\\ 0&0&0\end{array}\right),\\
M_{21}=\left(\begin{array}{ccc}0&0&0\\ 0&0&0\\ 1&0&0\end{array}\right),\,\,
M_{22}=\left(\begin{array}{ccc}0&0&0\\ 0&0&0\\ 0&1&1\end{array}\right).
\end{gather*}
\end{proof}

\section{Proof of Theorem~\ref{mainresult} for $\Bbbk(\bullet\to\bullet)$}\label{s7}

Let $\mathbf{M}$ be a simple  transitive $2$-representation of $\cC_A$.
Let $B$ be a basic finite dimensional algebra such that $\mathbf{M}(\mathtt{i})$
is equivalent to $B$-proj.

As the left $A$-modules $Ae_1$ and $\mathrm{Hom}_{\Bbbk}(e_2A,\Bbbk)$
are isomorphic, from Proposition~\ref{prop77-n} it follows that
the functor $\overline{\mathbf{M}}(\mathrm{F})$ is exact.
From Lemmata~\ref{lemmamm5} and \ref{lemmamm7} we thus obtain that
$\overline{\mathbf{M}}(\mathrm{F})$ is a projective endofunctor of $B$-mod.

{\bf Case 1.} Assume that $M=M_3$ and the $M_{ij}$'s are thus given by Proposition~\ref{prop2}\eqref{prop2.2}.
Let $\overline{\mathbf{M}}$ be the abelianization of $\mathbf{M}$. As usual, we write $P_1$ and $P_2$
for indecomposable projectives in $\mathbf{M}(\mathtt{i})$ and $L_1$ and $L_2$ for their respective simple tops.
Let $\epsilon_1$ and $\epsilon_2$ be the corresponding primitive idempotents in $B$. For $i,j=1,2$, denote by
$\mathrm{G}_{ij}$ an endofunctor of $\overline{\mathbf{M}}(\mathtt{i})$ which corresponds to tensoring
with $B\epsilon_i\otimes \epsilon_j B$.

From the form of $M_{21}$, we see that $\mathrm{F}_{21}$ acts via $\mathrm{G}_{12}$.
Similarly, $\mathrm{F}_{22}$ acts via $\mathrm{G}_{11}$.
From the matrices $M_{21}$ and $M_{22}$ it follows that
\begin{displaymath}
[P_1:L_1]=1,\quad
[P_1:L_2]=0,\quad
[P_2:L_1]=0,\quad
[P_2:L_2]=1.
\end{displaymath}
This means that $B\cong\Bbbk\oplus\Bbbk$. Therefore all $\mathrm{G}_{ij}$ are isomorphisms between the
corresponding $\Bbbk$-mod components. From the matrices $M_{12}$ and $M_{21}$ it thus follows directly that
there are no nonzero homomorphisms from $\mathrm{F}_{21}$ to $\mathrm{F}_{12}$. This contradicts \eqref{eq101} and
hence Case~1 cannot occur.

{\bf Case 2.} Assume that $M=M_2$ and the $M_{ij}$'s are thus given by Proposition~\ref{prop2}\eqref{prop2.1}.
Let $\overline{\mathbf{M}}$ be the abelianization of $\mathbf{M}$. As usual, we write $P_1$ and $P_2$
for indecomposable projectives in $\mathbf{M}(\mathtt{i})$ and $L_1$  and $L_2$ for their respective simple tops.
Let $\epsilon_1$  and $\epsilon_2$ be the corresponding primitive idempotents in $B$. For
$i,j=1,2$, denote by $\mathrm{G}_{ij}$ an endofunctor of $\overline{\mathbf{M}}(\mathtt{i})$
which corresponds to tensoring with $B\epsilon_i\otimes \epsilon_j B$.

From the form of $M_{11}$, we see that $\mathrm{F}_{11}$ acts via $\mathrm{G}_{11}$.
Similarly, $\mathrm{F}_{21}$ acts via $\mathrm{G}_{21}$.

From the form of $M_{12}$, we see that $\mathrm{F}_{12}$ acts either via $\mathrm{G}_{12}$
or via $\mathrm{G}_{11}$ or via $\mathrm{G}_{12}\oplus \mathrm{G}_{11}$. However, we already know
that the matrix of $\mathrm{G}_{11}$ is $M_{11}$. This leaves us with possibilities
$\mathrm{G}_{12}$ or $\mathrm{G}_{12}\oplus \mathrm{G}_{11}$ for $\mathrm{F}_{12}$.

Assume that $\mathrm{F}_{12}$ acts via $\mathrm{G}_{12}\oplus \mathrm{G}_{11}$.
We already know the matrix of $\mathrm{G}_{11}$, so the matrix of  $\mathrm{G}_{12}$ is
\begin{displaymath}
\left(\begin{array}{cc}0&1\\0&0\end{array}\right).
\end{displaymath}
This and the matrix $M_{11}$ imply that
\begin{displaymath}
\mathrm{G}_{11}\,P_1\cong P_1,\quad
\mathrm{G}_{11}\,P_2=0,\quad
\mathrm{G}_{12}\,P_1=0,\quad
\mathrm{G}_{12}\,P_2\cong P_2.
\end{displaymath}
Therefore
\begin{displaymath}
[P_1:L_1]=1,\quad
[P_1:L_2]=0,\quad
[P_2:L_1]=0,\quad
[P_2:L_2]=1
\end{displaymath}
and we have $B\cong\Bbbk\oplus\Bbbk$. This leads to the same contradiction as in Case~1 above.
Therefore $\mathrm{F}_{12}$ acts via $\mathrm{G}_{12}$. Similarly one shows that
$\mathrm{F}_{22}$ acts via $\mathrm{G}_{22}$.

From the matrices for all $\mathrm{G}_{ij}$'s it follows that the Cartan matrices of $A$ and $B$
coincide which implies that $A$ and $B$ are isomorphic (that is special for our case, but the algebra
$A$ is very small, so this claim is clear). Furthermore, all $\mathrm{F}_{ij}$'s act via the
corresponding $\mathrm{G}_{ij}$. It now
follows by the usual arguments, see \cite[Proposition~9]{MM5}, that
$\overline{\mathbf{M}}$ is equivalent to a cell $2$-representation of $\cC_A$.

{\bf Case 3.} Assume that $M=M_6$ and the $M_{ij}$'s are thus given by Proposition~\ref{prop2}\eqref{prop2.3}.
Let $\overline{\mathbf{M}}$ be the abelianization of $\mathbf{M}$. As usual, we write $P_1$, $P_2$ and $P_3$
for indecomposable projectives in $\mathbf{M}(\mathtt{i})$ and $L_1$, $L_2$ and $L_3$ for their respective simple tops.
Let $\epsilon_1$, $\epsilon_2$ and $\epsilon_3$ be the corresponding primitive idempotents in $B$. For
$i,j=1,2,3$, denote by $\mathrm{G}_{ij}$ an endofunctor of $\overline{\mathbf{M}}(\mathtt{i})$
which corresponds to tensoring with $B\epsilon_i\otimes \epsilon_j B$.

From the form of $M_{21}$, we see that $\mathrm{F}_{21}$ acts via $\mathrm{G}_{31}$.
From the form of $M_{11}$, we see that $\mathrm{F}_{11}$ acts via $\mathrm{G}_{11}\oplus \mathrm{G}_{21}$.
This implies
\begin{displaymath}
[P_1:L_1]=1,\quad  [P_2:L_1]=[P_3:L_1]=0.
\end{displaymath}

From the form of $M_{22}$, we see that $\mathrm{F}_{22}$ acts either  via $\mathrm{G}_{32}$
or via $\mathrm{G}_{33}$ or via $\mathrm{G}_{32}\oplus \mathrm{G}_{33}$.
In the latter case, we have that
the matrices of $\mathrm{G}_{32}$ and $\mathrm{G}_{33}$ are, respectively:
\begin{displaymath}
\left(\begin{array}{ccc}0&0&0\\0&0&0\\0&1&0\end{array}\right)\quad \text{ and }\quad
\left(\begin{array}{ccc}0&0&0\\0&0&0\\0&0&1\end{array}\right).
\end{displaymath}
It follows that
\begin{displaymath}
[P_2:L_2]=1,\quad  [P_1:L_2]=[P_3:L_2]=0
\end{displaymath}
and
\begin{displaymath}
[P_3:L_3]=1,\quad  [P_1:L_3]=[P_2:L_3]=0.
\end{displaymath}
This implies that $B\cong\Bbbk\oplus\Bbbk\oplus\Bbbk$ and leads to a similar contradiction as in Case~1.

{\bf Subcase 3.1.} Assume that $\mathrm{F}_{22}$ acts   via $\mathrm{G}_{32}$. This implies
\begin{displaymath}
[P_2:L_2]=[P_3:L_2]=1,\quad  [P_1:L_2]=0.
\end{displaymath}
From $[P_1:L_2]=0$ we have $\epsilon_2B\epsilon_1=0$. This means that
\begin{displaymath}
\mathrm{Hom}_{B\text{-}B}(B\epsilon_3\otimes \epsilon_1 B,B\epsilon_3\otimes \epsilon_2 B)=0,
\end{displaymath}
that is, $\mathrm{Hom}(\mathrm{G}_{31},\mathrm{G}_{32})=0$. This contradicts
$\mathrm{Hom}_{\ccC}(\mathrm{F}_{21},\mathrm{F}_{22})\neq 0$, see \eqref{eq101}.

{\bf Subcase 3.2.} Assume that $\mathrm{F}_{22}$ acts  via $\mathrm{G}_{33}$. This implies
\begin{displaymath}
[P_3:L_3]=[P_2:L_3]=1,\quad  [P_1:L_3]=0.
\end{displaymath}
From $[P_1:L_3]=0$ we have $\epsilon_3B\epsilon_1=0$. This means that
\begin{displaymath}
\mathrm{Hom}_{B\text{-}B}(B\epsilon_3\otimes \epsilon_1 B,B\epsilon_3\otimes \epsilon_3 B)=0,
\end{displaymath}
that is, $\mathrm{Hom}(\mathrm{G}_{31},\mathrm{G}_{33})=0$. This contradicts
$\mathrm{Hom}_{\ccC}(\mathrm{F}_{21},\mathrm{F}_{22})\neq 0$, see \eqref{eq101}.
The proof is now complete.

\section{The algebra $\Bbbk(\bullet\overset{\alpha}{\to} \bullet\overset{\beta}{\to}
\bullet)/(\beta\alpha)$}\label{s21}

Let $\Bbbk$ be an algebraically closed field. Denote by $A$ the path algebra, over $\Bbbk$, of the quiver
\begin{displaymath}
\Bbbk(\xymatrix{1\ar[r]^{\alpha}&2\ar[r]^{\beta}&3})\quad
\text{ modulo the relations }\quad\beta\alpha=0.
\end{displaymath}
The algebra $A$ has basis $e_1$, $e_2$, $e_3$, $\alpha$ and $\beta$
and the multiplication table $(x,y)\mapsto x\cdot y$ is given by:
\begin{displaymath}
\begin{array}{c||c|c|c|c|c}
x\backslash y & e_1 & e_2 & e_3 & \alpha & \beta\\
\hline\hline
e_1 & e_1 & 0 & 0& 0& 0\\
\hline
e_2 & 0 & e_2 & 0&  \alpha & 0\\
\hline
e_3 & 0 & 0 & e_3&  0 & \beta\\
\hline
\alpha & \alpha & 0 & 0& 0& 0\\
\hline
\beta & 0 & \beta & 0& 0& 0\\
\end{array}
\end{displaymath}
Note that $e_1Ae_2=0$, $e_1Ae_3=0$, $e_2Ae_3=0$ and $e_3Ae_1=0$.
Note also that the left $A$-modules $Ae_1$ and $\mathrm{Hom}_{\Bbbk}(e_2A,\Bbbk)$ are isomorphic
and the left $A$-modules $Ae_2$ and $\mathrm{Hom}_{\Bbbk}(e_3A,\Bbbk)$ are isomorphic.

Let $\mathcal{C}$ be a small category equivalent to $A$-mod. Consider the corresponding finitary
$2$-category $\cC_A$.
Up to isomorphism, indecomposable $1$-morphisms in $\cC_A$ are $\mathrm{F}_0$ and
$\mathrm{F}_{ij}$, where  $i,j=1,2,3$. Note that formula~\eqref{eqnn0} for $A$ reads
$\mathrm{F}\circ \mathrm{F}\cong \mathrm{F}^{\oplus 5}$. Using \eqref{eqnn1},
the table of compositions for the functors $\mathrm{F}_{ij}$ (up to isomorphism) is as follows:
\begin{equation}\label{eq1n}
\begin{array}{c||c|c|c|c|c|c|c|c|c}
\circ & \mathrm{F}_{11} & \mathrm{F}_{12} & \mathrm{F}_{13}&
\mathrm{F}_{21}& \mathrm{F}_{22}& \mathrm{F}_{23}&
\mathrm{F}_{31}& \mathrm{F}_{32}& \mathrm{F}_{33}\\
\hline\hline
\mathrm{F}_{11} & \mathrm{F}_{11} & \mathrm{F}_{12}& \mathrm{F}_{13} & 0& 0& 0& 0& 0& 0\\
\hline
\mathrm{F}_{12} & \mathrm{F}_{11} & \mathrm{F}_{12} & \mathrm{F}_{13}&
\mathrm{F}_{11} & \mathrm{F}_{12} & \mathrm{F}_{13}& 0&0&0\\
\hline
\mathrm{F}_{13} &0&0&0& \mathrm{F}_{11} & \mathrm{F}_{12} &\mathrm{F}_{13}
& \mathrm{F}_{11} & \mathrm{F}_{12} &\mathrm{F}_{13} \\
\hline
\mathrm{F}_{21} & \mathrm{F}_{21} & \mathrm{F}_{22}& \mathrm{F}_{23} & 0& 0& 0& 0& 0& 0\\
\hline
\mathrm{F}_{22} & \mathrm{F}_{21} & \mathrm{F}_{22} & \mathrm{F}_{23}&
\mathrm{F}_{21} & \mathrm{F}_{22} & \mathrm{F}_{23}& 0&0&0\\
\hline
\mathrm{F}_{23} &0&0&0& \mathrm{F}_{21} & \mathrm{F}_{22} &\mathrm{F}_{23}
& \mathrm{F}_{21} & \mathrm{F}_{22} &\mathrm{F}_{23} \\
\hline
\mathrm{F}_{31} & \mathrm{F}_{31} & \mathrm{F}_{32}& \mathrm{F}_{33} & 0& 0& 0& 0& 0& 0\\
\hline
\mathrm{F}_{32} & \mathrm{F}_{31} & \mathrm{F}_{32} & \mathrm{F}_{33}&
\mathrm{F}_{31} & \mathrm{F}_{32} & \mathrm{F}_{33}& 0&0&0\\
\hline
\mathrm{F}_{33} &0&0&0& \mathrm{F}_{31} & \mathrm{F}_{32} &\mathrm{F}_{33}
& \mathrm{F}_{31} & \mathrm{F}_{32} &\mathrm{F}_{33} \\
\end{array}
\end{equation}
Set $\mathcal{J}_0:=\{\mathrm{F}_{0}\}$ and $\mathcal{J}:=\{\mathrm{F}_{ij}\,:\,i,j=1,2,3\}$.
Note that the $2$-category $\cC_A$ is not weakly fiat in the sense of \cite{MM2,MM6} as
the algebra $A$ is not self-injective.

As $\cC_A$ is $\mathcal{J}$-simple and $A$ has trivial center, the only proper non-zero
quotient of  $\cC_A$ contains just the identity $1$-morphism (up to isomorphism) and its scalar
endomorphisms (cf. \cite{MM3}). Therefore this quotient is fiat with strongly regular $\mathcal{J}$-classes and hence
it has a unique, up to equivalence, simple transitive $2$-representation, namely $\mathbf{C}_{\mathcal{L}_0}$,
where $\mathcal{L}_0=\mathcal{J}_0$, see \cite[Theorem~18]{MM5}. This means that, in order to prove
Theorem~\ref{mainresult} for $A$, it is enough to consider {\em faithful} $2$-representations of $\cC_A$.

From \eqref{eq10n}, we get the following table of
$\mathrm{Hom}_{\ccC_{\hspace{-1mm}A}(\mathtt{i})}(X,Y)$ (up to isomorphism), where $X$ and $Y$ are
indecomposable $1$-morphisms as specified in the table:
\begin{equation}\label{eq101nn}
\begin{array}{c||c|c|c|c|c|c|c|c|c}
X\setminus Y & \mathrm{F}_{11} & \mathrm{F}_{12} & \mathrm{F}_{13}& \mathrm{F}_{21}& \mathrm{F}_{22}&
\mathrm{F}_{23}& \mathrm{F}_{31}& \mathrm{F}_{32}& \mathrm{F}_{33}\\
\hline\hline
\mathrm{F}_{11} & \Bbbk & \Bbbk & 0 & 0& 0& 0& 0& 0& 0\\
\hline
\mathrm{F}_{12} & 0& \Bbbk & \Bbbk& 0& 0& 0& 0& 0& 0\\
\hline
\mathrm{F}_{13} & 0 & 0 & \Bbbk& 0& 0& 0& 0& 0& 0\\
\hline
\mathrm{F}_{21} & \Bbbk & \Bbbk & 0& \Bbbk & \Bbbk & 0& 0& 0& 0\\
\hline
\mathrm{F}_{22} & 0 & \Bbbk & \Bbbk& 0& \Bbbk & \Bbbk & 0& 0& 0\\
\hline
\mathrm{F}_{23} & 0 & 0 & \Bbbk& 0& 0& \Bbbk & 0& 0& 0\\
\hline
\mathrm{F}_{31} & 0 & 0 & 0& \Bbbk& \Bbbk& 0& \Bbbk& \Bbbk& 0\\
\hline
\mathrm{F}_{32} & 0 & 0 & 0& 0&\Bbbk& \Bbbk&  0& \Bbbk& \Bbbk\\
\hline
\mathrm{F}_{33} & 0 & 0 & 0& 0&0& \Bbbk&  0& 0& \Bbbk
\end{array}
\end{equation}

\section{Integer matrices for $\Bbbk(\bullet\overset{\alpha}{\to}
\bullet\overset{\beta}{\to} \bullet)/(\beta\alpha)$}\label{s11}

\subsection{Integer matrices satisfying $M^2=5M$}\label{s11.1}

In this section we classify all square matrices $M$ with positive integer coefficients which satisfy $M^2=5M$.

\begin{proposition}\label{prop101}
Let $M$ be a $k\times k$ matrix, for some $k$, with positive integer coefficients, satisfying $M^2=5M$. Then,
up to permutation action, $M$ is one of the following matrices:
\begin{displaymath}
N_1:=\left(5\right),\quad
N_2:=\left(\begin{array}{cc}4&1\\4&1\end{array}\right),\quad
N_3:=\left(\begin{array}{cc}4&4\\1&1\end{array}\right),\quad
N_4:=\left(\begin{array}{cc}4&2\\2&1\end{array}\right),
\end{displaymath}
\begin{displaymath}
N_5:=\left(\begin{array}{cc}3&6\\1&2\end{array}\right),\quad
N_6:=\left(\begin{array}{cc}3&3\\2&2\end{array}\right),\quad
N_7:=\left(\begin{array}{cc}3&2\\3&2\end{array}\right),\quad
N_8:=\left(\begin{array}{cc}3&1\\6&2\end{array}\right),
\end{displaymath}
\begin{displaymath}
N_9:=\left(\begin{array}{ccc}3&1&1\\3&1&1\\3&1&1\end{array}\right),\quad
N_{10}:=\left(\begin{array}{ccc}3&3&3\\1&1&1\\1&1&1\end{array}\right),\quad
N_{11}:=\left(\begin{array}{ccc}2&2&2\\2&2&2\\1&1&1\end{array}\right),
\end{displaymath}
\begin{displaymath}
N_{12}:=\left(\begin{array}{ccc}2&4&2\\1&2&1\\1&2&1\end{array}\right),\quad
N_{13}:=\left(\begin{array}{ccc}2&2&1\\2&2&1\\2&2&1\end{array}\right),\quad
N_{14}:=\left(\begin{array}{ccccc}1&1&1&1&1\\1&1&1&1&1\\1&1&1&1&1\\1&1&1&1&1\\1&1&1&1&1\end{array}\right),
\end{displaymath}
\begin{displaymath}
N_{15}:=\left(\begin{array}{cccc}2&1&1&1\\2&1&1&1\\2&1&1&1\\2&1&1&1\end{array}\right),\quad
N_{16}:=\left(\begin{array}{cccc}2&2&2&2\\1&1&1&1\\1&1&1&1\\1&1&1&1\end{array}\right).
\end{displaymath}
\end{proposition}

We note an important difference with Proposition~\ref{prop1}: to make our list shorter,
Proposition~\ref{prop101} gives classification only up to permutation action.

\begin{proof}
Clearly, we have $N_i^2=5N_i$, for each $i=1,2,\dots,16$. So, we need to show that any other square matrix
with positive integer coefficients satisfying $M^2=5M$ can be reduced to one of the above using
permutation action.

Let $M$ be a $k\times k$ matrix, for some $k$, with positive integer coefficients satisfying $M^2=5M$.
Then $M$ is diagonalizable (as $x^2-5x$ has no multiple roots) and the only possible eigenvalues
for $M$ are $0$ and $5$.  From the Perron-Frobenius theorem it follows
that the Perron-Frobenius eigenvalue $5$ must have multiplicity one. Therefore $M$ has rank one and trace five.
As all entries in $M$ are positive integers, we get $k\leq 5$. Using the permutation action, we may
assume that the entries on the main diagonal of $M$ weakly decrease from the top left corner to the
bottom right corner.

If $k=1$, then, clearly,  $M=N_1$.

If $k=2$, then the diagonal of  $M$ is either $(4,1)$ or $(3,2)$. In the first case, as
the determinant of $M$ is zero, the two remaining entries are either $2$ and $2$ or $4$ and $1$.
This gives $M=N_2$, $M=N_3$ or $M=N_4$. In the second case, as  the determinant of $M$ is zero,
the two remaining entries are either $2$ and $3$ or $1$ and $6$.
This gives $M=N_5$, $M=N_6$, $M=N_7$ or $M=N_8$.

If $k=3$, then the diagonal of  $M$ is either $(3,1,1)$ or $(2,2,1)$. In the first case,
as  $M$ has rank one, any $2\times 2$ minor in $M$ has determinant zero. This means that all
entries which are neither in the first row nor in the first column are equal to $1$.
If the first row contains more than one entry different from $1$, then all entires
in this row are $3$ and we get $M=N_{10}$. If the first column contains more than one entry different
from $1$, then all  entires in this column are $3$ and we get $M=N_{9}$.

In the second case, write
\begin{displaymath}
M=\left(\begin{array}{ccc}2&m_{12}&m_{13}\\m_{21}&2&m_{23}\\m_{31}&m_{32}&1\end{array}\right).
\end{displaymath}
Then $m_{32}m_{23}=2$, $m_{31}m_{13}=2$ and $m_{21}m_{12}=4$. Hence, both $(m_{32},m_{23})$
and $(m_{31},m_{13})$ are in $\{(1,2),(2,1)\}$. We can choose them independently and the fact
that $M$ has rank one then uniquely determines the pair $(m_{21},m_{12})$. This gives us
$M=N_{11}$, $M=N_{12}$ and $M=N_{13}$ and also the possibility
\begin{displaymath}
M=N'_{12}:=\left(\begin{array}{ccc}2&1&1\\4&2&2\\2&1&1\end{array}\right).
\end{displaymath}
which reduces to $M=N_{12}$ by permutation action.

If $k=4$, then the diagonal of  $M$ is $(2,1,1,1)$. As  $M$ has rank one, any $2\times 2$ minor
in $M$ has determinant zero. This means that all entries which are neither in the first row nor
in the first column are equal to $1$. If the first row contains more than one entry different from $1$,
then all entires in this row are $2$ and we get $M=N_{16}$. If the first column contains more than
one entry different from $1$, then all  entires in this column are $2$ and we get $M=N_{15}$.

If $k=5$, then all diagonal entries in $M$ are $1$. As all $2\times 2$ minors in $M$ should have determinant
zero and positive integer entries, it follows that all entries in $M$ are $1$ and thus $M=N_{14}$.
\end{proof}

\subsection{Filtering  ``easy cases'' out}\label{s11.2}

Let $\mathbf{M}$ be a finitary, simple, transitive and faithful $2$-representation of $\cC_A$.
Let $M:=[\mathrm{F}]$ be the matrix of $\mathbf{M}(\mathrm{F})$ and, for $i,j=1,2,3$, let
$M_{ij}:=[\mathrm{F}_{ij}]$ be the matrix of $\mathbf{M}(\mathrm{F}_{ij})$. We have
$M_{ij}\neq 0$, for all $i,j=1,2,3$. By Proposition~\ref{prop101}, up to permutation action,
we have $M=N_i$, for some $i\in\{1,2,\dots,16\}$ as in Proposition~\ref{prop101}. Note that
trace of $M$ is five.

As usual, we call ``position  $(i,j)$'' the
intersection of the $i$-th row and  the $j$-th column of a matrix.

From now on, we assume that $\overline{\mathbf{M}}(\mathtt{i})$ is equivalent to $B$-mod,
for some basic algebra $B$. Let $L_1$, $L_2$,\dots, $L_k$ be a complete and
irredundant list of representatives of isomorphism classes of simple objects in
$\overline{\mathbf{M}}(\mathtt{i})$. For $i\in\{1,2,\dots,k\}$,  denote by $P_i$ the
indecomposable projective cover of $L_i$ and by $I_i$ the indecomposable injective envelope of $L_i$.
The matrices $M_{ij}$ are given with respect to this fixed ordering of isomorphism classes of
simple objects.

\begin{lemma}\label{lem201}
{\hspace{2mm}}

\begin{enumerate}[$($i$)$]
\item\label{lem201.1} All diagonal elements in $M_{13}$, $M_{21}$, $M_{31}$ and $M_{32}$ are zero.
\item\label{lem201.2} Each of the matrices $M_{11}$, $M_{12}$, $M_{22}$, $M_{23}$ and $M_{33}$,
has exactly one entry equal to $1$ on the diagonal and all other diagonal entries are zero.
\end{enumerate}
\end{lemma}

\begin{proof}
From \eqref{eq1n}, we see that $\mathrm{F}_{ij}$ is idempotent if and only if
\begin{displaymath}
(i,j)\in\{(1,1),(1,2),(2,2),(2,3),(3,3)\}
\end{displaymath}
and $\mathrm{F}_{ij}^2=0$ otherwise. As the trace of a non-zero idempotent with
non-negative coefficients is non-zero, each idempotent $M_{ij}$
has trace at least one. As trace of $M$ is five, it follows that all idempotents $M_{ij}$ have
trace one. This proves claim~\eqref{lem201.2}. Claim~\eqref{lem201.1} follows from
claim~\eqref{lem201.2} as $M$ is the sum of the $M_{ij}$'s.
\end{proof}

\begin{corollary}\label{cor203}
The matrix $M$ cannot be equal to $N_i$, where $i\in\{1,2,\dots,10\}$.
\end{corollary}

\begin{proof}
Each of the matrices $N_i$, where $i\in\{1,2,\dots,10\}$, contains a diagonal element which is greater than
or equal to $3$. If $M=N_i$ would be possible, at least three idempotents $M_{ij}$ would have this
diagonal element non-zero. But then any product of any two such matrices would be non-zero.
However, from \eqref{eq1n} we have that, for any three different idempotents $\mathrm{F}_{ij}$,
one of the products of two of these elements is zero. The obtained contradiction completes the proof.
\end{proof}

\subsection{Auxiliary adjunction}\label{s12.1}

We will need the following easy observations:

\begin{lemma}\label{lem211}
Let $D$ be a finite dimensional algebra and $(\mathrm{G},\mathrm{H})$ an adjoint pair of right
exact endofunctors of $D$-mod. Let $L$ and $L'$ be simple $D$-modules and $P$ and  $P'$ their
corresponding indecomposable projective covers. Assume that $L'$ appears in the top of $\mathrm{G}\, P$.
Then $\mathrm{H}\,P'\neq 0$.
\end{lemma}

\begin{proof}
By adjunction, we have
\begin{displaymath}
0\neq\mathrm{Hom}_B(\mathrm{G}\, P,L')\cong  \mathrm{Hom}_B(P,\mathrm{H}\, L'),
\end{displaymath}
which implies $\mathrm{H}\, L'\neq 0$. As $\mathrm{H}$ is right exact, this forces
$\mathrm{H}\,P'\neq 0$.
\end{proof}

\begin{lemma}\label{lem212}
We have the following pairs of adjoint $1$-morphisms in $\cC_A$:
\begin{displaymath}
(\mathrm{F}_{33},\mathrm{F}_{23}),\,\,\,\, (\mathrm{F}_{23},\mathrm{F}_{22}),\,\,\,\,
(\mathrm{F}_{22},\mathrm{F}_{12}),\,\,\,\, (\mathrm{F}_{12},\mathrm{F}_{11}).
\end{displaymath}

\end{lemma}

\begin{proof}
As both,  the left $A$-modules $Ae_1$ and $\mathrm{Hom}_{\Bbbk}(e_2A,\Bbbk)$ are isomorphic
and the left $A$-modules $Ae_2$ and $\mathrm{Hom}_{\Bbbk}(e_3A,\Bbbk)$ are isomorphic,
the claim follows from Lemma~\ref{lem72-n}.
\end{proof}

\subsection{Idempotent integral matrices of rank one}\label{s12.2}

Recall from \cite[Theorem~2]{Fl} that, up to permutation action, idempotent matrices of rank one
with non-negative integer entries have the form
\begin{displaymath}
\left(\begin{array}{ccc}\mathbf{0}&v&vw^t\\
\mathbf{0}&1&w^t\\\mathbf{0}&\mathbf{0}&\mathbf{0}
\end{array}\right)
\end{displaymath}
where $\mathbf{0}$ denotes the zero matrix (of an appropriate size), and $v$ and $w$ are arbitrary vectors
with non-negative integer entries. In particular, if the diagonal entry $1$ is in the $i$-th row,
then the whole matrix can be written as the product of its $i$-th column with its $i$-th row.

\subsection{Filtering matrices $N_{14}$, $N_{15}$ and $N_{16}$ out}\label{s12.3}

\begin{proposition}\label{prop231}
The matrix $M$ cannot be equal to $N_{14}$, $N_{15}$ or $N_{16}$.
\end{proposition}

\begin{proof}
Assume that the diagonals of the matrices $M_{11}$ and $M_{12}$ are different. This means that
$M_{11}$ has $1$ in row $i$, that $M_{12}$ has $1$ in row $j$, and that $i\neq j$.

Then $\mathrm{F}_{12}\, P_j$ has $P_j$ as a direct summand. Therefore, by
combining Lemmata~\ref{lem211} and \ref{lem212}, we have that $M_{11}$
must have a non-zero element in column $j$. From Subsection~\ref{s12.2}
it follows that $M_{11}$ has a non-zero entry in position $(i,j)$.

As $M_{11}$ has a non-zero entry in position $(i,j)$ and the matrix $M_{12}$ has $1$ in position $(j,j)$,
it follows that $M_{11}M_{12}$ has a non-zero entry in position $(i,j)$. From \eqref{eq1n},
we have $M_{11}M_{12}=M_{12}$, which means that $M_{12}$ has a non-zero entry in position $(i,j)$.
This already means that the case $M=N_{14}$ is not possible.

Assume $M=N_{15}$. Then we must have $j=1$. Exactly the same argument as above applied to
$M_{22}$ and $M_{23}$ shows that $M_{23}$ has $1$ in position $(1,1)$. This contradicts
$M_{23}M_{12}=0$ as follows from \eqref{eq1n}. Therefore $M=N_{15}$ is not possible.

Assume $M=N_{16}$. Then we must have $i=1$. Exactly the same argument as above applied to
$M_{22}$ and $M_{23}$ shows that $M_{22}$ has $1$ in position $(1,1)$. This contradicts
$M_{11}M_{22}=0$ as follows from \eqref{eq1n}. Therefore $M=N_{16}$ is not possible. This completes the proof.
\end{proof}

The remaining cases for $M$ will be studied on a case-by-case basis.

\section{Filtering matrices $N_{11}$ and  $N_{12}$ out}\label{s12n}

\subsection{Statement}\label{s12.5}

The main aim of this section is to prove the following

\begin{proposition}\label{prop251}
The matrix $M$ cannot be equal to $N_{11}$ or $N_{12}$.
\end{proposition}

We start with the following observation.

\begin{lemma}\label{lem273}
The only unordered pairs of idempotent $1$-morphisms of the form $\mathrm{F}_{ij}$ such that
the product of any two elements in the pair is non-zero are
\begin{displaymath}
\{\mathrm{F}_{11},\mathrm{F}_{12}\},\quad
\{\mathrm{F}_{12},\mathrm{F}_{22}\},\quad
\{\mathrm{F}_{22},\mathrm{F}_{23}\},\quad
\{\mathrm{F}_{23},\mathrm{F}_{33}\}.
\end{displaymath}
\end{lemma}

\begin{proof}
This follows directly from \eqref{eq1n}.
\end{proof}

\subsection{Proof for $M=N_{11}$}\label{s12.6}

We will arrange matrices $M_{ij}$, where $i,j=1,2,3$, as follows:
\begin{equation}\label{eq773}
\begin{array}{ccc}
M_{11}&M_{12}&M_{13}\\
M_{21}&M_{22}&M_{23}\\
M_{31}&M_{32}&M_{33}.
\end{array}
\end{equation}
Assume $M=N_{11}$. The diagonal elements in $N_{11}$ are $(2,2,1)$. Therefore,
two pairs of idempotent matrices of the form $M_{ij}$ would have common
diagonal elements. Any product of matrices in any such pair would be non-zero.
Therefore, using  Lemma~\ref{lem273}, we have three cases to consider.

{\bf Case~1.} Suppose first that the pairs of idempotent matrices which share
diagonal elements are $\{\mathrm{F}_{11},\mathrm{F}_{12}\}$ and $\{\mathrm{F}_{22},\mathrm{F}_{23}\}$.
Up to permutation action, we may assume that
\begin{displaymath}
\left(\begin{array}{ccc}1&*&*\\ *&0&*\\ *&*&0\end{array}\right),\quad
\left(\begin{array}{ccc}1&*&*\\ *&0&*\\ *&*&0\end{array}\right),\quad
\left(\begin{array}{ccc}0&*&*\\ *&0&*\\ *&*&0\end{array}\right),
\end{displaymath}
\begin{displaymath}
\left(\begin{array}{ccc}0&*&*\\ *&0&*\\ *&*&0\end{array}\right),\quad
\left(\begin{array}{ccc}0&*&*\\ *&1&*\\ *&*&0\end{array}\right),\quad
\left(\begin{array}{ccc}0&*&*\\ *&1&*\\ *&*&0\end{array}\right),
\end{displaymath}
\begin{displaymath}
\left(\begin{array}{ccc}0&*&*\\ *&0&*\\ *&*&0\end{array}\right),\quad
\left(\begin{array}{ccc}0&*&*\\ *&0&*\\ *&*&0\end{array}\right),\quad
\left(\begin{array}{ccc}0&*&*\\ *&0&*\\ *&*&1\end{array}\right).
\end{displaymath}
Using all possible zero products which appear in \eqref{eq1n}, we obtain that the $M_{ij}$'s look as follows:
\begin{displaymath}
\left(\begin{array}{ccc}1&0&0\\ 0&0&0\\ 0&0&0\end{array}\right),\quad
\left(\begin{array}{ccc}1&*&0\\ 0&0&0\\ 0&0&0\end{array}\right),\quad
\left(\begin{array}{ccc}0&*&*\\ 0&0&0\\ 0&0&0\end{array}\right),
\end{displaymath}
\begin{displaymath}
\left(\begin{array}{ccc}0&0&0\\ *&0&0\\ *&0&0\end{array}\right),\quad
\left(\begin{array}{ccc}0&0&0\\ *&1&0\\ *&*&0\end{array}\right),\quad
\left(\begin{array}{ccc}0&0&0\\ 0&1&*\\ 0&*&0\end{array}\right),
\end{displaymath}
\begin{displaymath}
\left(\begin{array}{ccc}0&0&0\\ 0&0&0\\ *&0&0\end{array}\right),\quad
\left(\begin{array}{ccc}0&0&0\\ 0&0&0\\ *&*&0\end{array}\right),\quad
\left(\begin{array}{ccc}0&0&0\\ 0&0&0\\ 0&*&1\end{array}\right).
\end{displaymath}
As all matrices must be non-zero and add up to $M$, we obtain
\begin{displaymath}
\left(\begin{array}{ccc}1&0&0\\ 0&0&0\\ 0&0&0\end{array}\right),\quad
\left(\begin{array}{ccc}1&*&0\\ 0&0&0\\ 0&0&0\end{array}\right),\quad
\left(\begin{array}{ccc}0&*&2\\ 0&0&0\\ 0&0&0\end{array}\right),
\end{displaymath}
\begin{displaymath}
\left(\begin{array}{ccc}0&0&0\\ *&0&0\\ 0&0&0\end{array}\right),\quad
\left(\begin{array}{ccc}0&0&0\\ *&1&0\\ 0&0&0\end{array}\right),\quad
\left(\begin{array}{ccc}0&0&0\\ 0&1&2\\ 0&0&0\end{array}\right),
\end{displaymath}
\begin{displaymath}
\left(\begin{array}{ccc}0&0&0\\ 0&0&0\\ 1&0&0\end{array}\right),\quad
\left(\begin{array}{ccc}0&0&0\\ 0&0&0\\ 0&1&0\end{array}\right),\quad
\left(\begin{array}{ccc}0&0&0\\ 0&0&0\\ 0&0&1\end{array}\right).
\end{displaymath}
This contradicts $M_{13}M_{31}=M_{11}$ which is a
consequence of \eqref{eq1n}. Therefore this case is not possible.

{\bf Case~2.} Suppose now that the pairs of idempotent matrices which share
diagonal elements are $\{\mathrm{F}_{12},\mathrm{F}_{22}\}$ and $\{\mathrm{F}_{23},\mathrm{F}_{33}\}$.
Up to permutation action, we may assume that the $M_{ij}$'s look as follows:
\begin{displaymath}
\left(\begin{array}{ccc}0&*&*\\ *&0&*\\ *&*&1\end{array}\right),\quad
\left(\begin{array}{ccc}1&*&*\\ *&0&*\\ *&*&0\end{array}\right),\quad
\left(\begin{array}{ccc}0&*&*\\ *&0&*\\ *&*&0\end{array}\right),
\end{displaymath}
\begin{displaymath}
\left(\begin{array}{ccc}0&*&*\\ *&0&*\\ *&*&0\end{array}\right),\quad
\left(\begin{array}{ccc}1&*&*\\ *&0&*\\ *&*&0\end{array}\right),\quad
\left(\begin{array}{ccc}0&*&*\\ *&1&*\\ *&*&0\end{array}\right),
\end{displaymath}
\begin{displaymath}
\left(\begin{array}{ccc}0&*&*\\ *&0&*\\ *&*&0\end{array}\right),\quad
\left(\begin{array}{ccc}0&*&*\\ *&0&*\\ *&*&0\end{array}\right),\quad
\left(\begin{array}{ccc}0&*&*\\ *&1&*\\ *&*&0\end{array}\right).
\end{displaymath}
Using all possible zero products which appear in \eqref{eq1n}, we obtain that the $M_{ij}$'s look as follows:
\begin{displaymath}
\left(\begin{array}{ccc}0&0&*\\ 0&0&0\\ 0&0&1\end{array}\right),\quad
\left(\begin{array}{ccc}1&0&*\\ 0&0&0\\ *&0&0\end{array}\right),\quad
\left(\begin{array}{ccc}0&*&0\\ 0&0&0\\ 0&*&0\end{array}\right),
\end{displaymath}
\begin{displaymath}
\left(\begin{array}{ccc}0&0&*\\ 0&0&*\\ 0&0&0\end{array}\right),\quad
\left(\begin{array}{ccc}1&0&*\\ *&0&*\\ 0&0&0\end{array}\right),\quad
\left(\begin{array}{ccc}0&*&0\\ 0&1&0\\ 0&0&0\end{array}\right),
\end{displaymath}
\begin{displaymath}
\left(\begin{array}{ccc}0&0&0\\ 0&0&*\\ 0&0&0\end{array}\right),\quad
\left(\begin{array}{ccc}0&0&0\\ *&0&*\\ 0&0&0\end{array}\right),\quad
\left(\begin{array}{ccc}0&0&0\\ 0&1&0\\ 0&0&0\end{array}\right).
\end{displaymath}
Here we have that $P_2$ is a direct summand of $\mathrm{F}_{23}\, P_2$.
From Lemmata~\ref{lem211} and \ref{lem212},
it follows that $\mathrm{F}_{22}\, P_2\neq 0$, which is a contradiction.
Therefore this case cannot occur either.

{\bf Case~3.} Suppose first that the pairs of idempotent matrices which share
diagonal elements are $\{\mathrm{F}_{11},\mathrm{F}_{12}\}$ and $\{\mathrm{F}_{23},\mathrm{F}_{33}\}$.
Up to permutation action, we may assume that the $M_{ij}$'s look as follows:
\begin{displaymath}
\left(\begin{array}{ccc}1&*&*\\ *&0&*\\ *&*&0\end{array}\right),\quad
\left(\begin{array}{ccc}1&*&*\\ *&0&*\\ *&*&0\end{array}\right),\quad
\left(\begin{array}{ccc}0&*&*\\ *&0&*\\ *&*&0\end{array}\right),
\end{displaymath}
\begin{displaymath}
\left(\begin{array}{ccc}0&*&*\\ *&0&*\\ *&*&0\end{array}\right),\quad
\left(\begin{array}{ccc}0&*&*\\ *&0&*\\ *&*&1\end{array}\right),\quad
\left(\begin{array}{ccc}0&*&*\\ *&1&*\\ *&*&0\end{array}\right),
\end{displaymath}
\begin{displaymath}
\left(\begin{array}{ccc}0&*&*\\ *&0&*\\ *&*&0\end{array}\right),\quad
\left(\begin{array}{ccc}0&*&*\\ *&0&*\\ *&*&0\end{array}\right),\quad
\left(\begin{array}{ccc}0&*&*\\ *&1&*\\ *&*&0\end{array}\right).
\end{displaymath}
Using all possible zero products which appear in \eqref{eq1n}, we obtain that the $M_{ij}$'s look as follows:
\begin{displaymath}
\left(\begin{array}{ccc}1&0&0\\ 0&0&0\\ *&0&0\end{array}\right),\quad
\left(\begin{array}{ccc}1&0&*\\ 0&0&0\\ *&0&0\end{array}\right),\quad
\left(\begin{array}{ccc}0&*&*\\ 0&0&0\\ 0&*&0\end{array}\right),
\end{displaymath}
\begin{displaymath}
\left(\begin{array}{ccc}0&0&0\\ *&0&0\\ *&0&0\end{array}\right),\quad
\left(\begin{array}{ccc}0&0&0\\ *&0&*\\ *&0&1\end{array}\right),\quad
\left(\begin{array}{ccc}0&0&0\\ 0&1&*\\ 0&*&0\end{array}\right),
\end{displaymath}
\begin{displaymath}
\left(\begin{array}{ccc}0&0&0\\ *&0&0\\ 0&0&0\end{array}\right),\quad
\left(\begin{array}{ccc}0&0&0\\ *&0&*\\ 0&0&0\end{array}\right),\quad
\left(\begin{array}{ccc}0&0&0\\ 0&1&*\\ 0&0&0\end{array}\right).
\end{displaymath}
Here we have that $P_2$ is a direct summand of $\mathrm{F}_{23}\, P_2$.
From Lemmata~\ref{lem211} and \ref{lem212},
it follows that $\mathrm{F}_{22}\, P_2\neq 0$, which is a contradiction.
Therefore this case cannot occur either.

This completes the proof of Lemma~\ref{lem273}, for $M=N_{11}$.

\subsection{Proof for $M=N_{12}$}\label{s12.7}

Let
\begin{displaymath}
N'_{12}:=\left(\begin{array}{ccc}2&1&1\\4&2&2\\2&1&1\end{array}\right).
\end{displaymath}
The matrix $N'_{12}$ reduces to $M=N_{12}$ by permutation action,
however, it is convenient to use the freedom of permutation action in another way,
see below. Because of  Lemma~\ref{lem273}, we have three cases to consider.

{\bf Case~1.} Suppose first that the pairs of idempotent matrices which share
diagonal elements are $\{\mathrm{F}_{11},\mathrm{F}_{12}\}$ and $\{\mathrm{F}_{22},\mathrm{F}_{23}\}$.
Then, using permutation action and  all possible zero products which appear in
\eqref{eq1n}, we obtain that the $M_{ij}$'s look as follows:
\begin{displaymath}
\left(\begin{array}{ccc}1&0&0\\ 0&0&0\\ 0&0&0\end{array}\right),\quad
\left(\begin{array}{ccc}1&*&0\\ 0&0&0\\ 0&0&0\end{array}\right),\quad
\left(\begin{array}{ccc}0&*&*\\ 0&0&0\\ 0&0&0\end{array}\right),
\end{displaymath}
\begin{displaymath}
\left(\begin{array}{ccc}0&0&0\\ *&0&0\\ *&0&0\end{array}\right),\quad
\left(\begin{array}{ccc}0&0&0\\ *&1&0\\ *&*&0\end{array}\right),\quad
\left(\begin{array}{ccc}0&0&0\\ 0&1&*\\ 0&*&0\end{array}\right),
\end{displaymath}
\begin{displaymath}
\left(\begin{array}{ccc}0&0&0\\ 0&0&0\\ *&0&0\end{array}\right),\quad
\left(\begin{array}{ccc}0&0&0\\ 0&0&0\\ *&*&0\end{array}\right),\quad
\left(\begin{array}{ccc}0&0&0\\ 0&0&0\\ 0&*&1\end{array}\right).
\end{displaymath}
If $M=N_{12}$, then the $(1,3)$-entry of $M_{13}$ equals $2$
while the $(3,1)$-entry of $M_{31}$ equals $1$. As the $(1,1)$-entry of $M_{11}$ is $1$,
we get a contradiction to $M_{13}M_{31}=M_{11}$, which follows from \eqref{eq1n}. This
implies that $M=N'_{12}$ and, using also $M_{12}M_{21}=M_{11}$, we have:
\begin{displaymath}
\left(\begin{array}{ccc}1&0&0\\ 0&0&0\\ 0&0&0\end{array}\right),\quad
\left(\begin{array}{ccc}1&1&0\\ 0&0&0\\ 0&0&0\end{array}\right),\quad
\left(\begin{array}{ccc}0&0&1\\ 0&0&0\\ 0&0&0\end{array}\right),
\end{displaymath}
\begin{displaymath}
\left(\begin{array}{ccc}0&0&0\\ 1&0&0\\ *&0&0\end{array}\right),\quad
\left(\begin{array}{ccc}0&0&0\\ *&1&0\\ *&*&0\end{array}\right),\quad
\left(\begin{array}{ccc}0&0&0\\ 0&1&*\\ 0&*&0\end{array}\right),
\end{displaymath}
\begin{displaymath}
\left(\begin{array}{ccc}0&0&0\\ 0&0&0\\ 1&0&0\end{array}\right),\quad
\left(\begin{array}{ccc}0&0&0\\ 0&0&0\\ *&*&0\end{array}\right),\quad
\left(\begin{array}{ccc}0&0&0\\ 0&0&0\\ 0&*&1\end{array}\right).
\end{displaymath}
As the $M_{ij}$'s must add up to $M$, we have
\begin{displaymath}
\left(\begin{array}{ccc}1&0&0\\ 0&0&0\\ 0&0&0\end{array}\right),\quad
\left(\begin{array}{ccc}1&1&0\\ 0&0&0\\ 0&0&0\end{array}\right),\quad
\left(\begin{array}{ccc}0&0&1\\ 0&0&0\\ 0&0&0\end{array}\right),
\end{displaymath}
\begin{displaymath}
\left(\begin{array}{ccc}0&0&0\\ 1&0&0\\ *&0&0\end{array}\right),\quad
\left(\begin{array}{ccc}0&0&0\\ 3&1&0\\ *&*&0\end{array}\right),\quad
\left(\begin{array}{ccc}0&0&0\\ 0&1&2\\ 0&*&0\end{array}\right),
\end{displaymath}
\begin{displaymath}
\left(\begin{array}{ccc}0&0&0\\ 0&0&0\\ 1&0&0\end{array}\right),\quad
\left(\begin{array}{ccc}0&0&0\\ 0&0&0\\ *&*&0\end{array}\right),\quad
\left(\begin{array}{ccc}0&0&0\\ 0&0&0\\ 0&*&1\end{array}\right).
\end{displaymath}
This contradicts $M_{23}M_{31}=M_{21}$, which follows from \eqref{eq1n}.
Therefore this case cannot occur.

{\bf Case~2.} Suppose now that the pairs of idempotent matrices which share
diagonal elements are $\{\mathrm{F}_{12},\mathrm{F}_{22}\}$ and $\{\mathrm{F}_{23},\mathrm{F}_{33}\}$.
This gives the same contradiction as in Case~2 in Subsection~\ref{s12.6}.
Therefore this case cannot occur either.

{\bf Case~3.} Suppose first that the pairs of idempotent matrices which share
diagonal elements are $\{\mathrm{F}_{11},\mathrm{F}_{12}\}$ and $\{\mathrm{F}_{23},\mathrm{F}_{33}\}$.
This gives the same contradiction as in Case~3 in Subsection~\ref{s12.6}.
Therefore this case cannot occur either.
This completes the proof of Lemma~\ref{lem273}.

\section{Proof of Theorem~\ref{mainresult} for
$\Bbbk(\bullet\overset{\alpha}{\to} \bullet\overset{\beta}{\to} \bullet)/(\beta\alpha)$}\label{s19}

\subsection{Finding the matrices}\label{s19.1}

Combining Proposition~\ref{prop101} with  Corollary~\ref{cor203}, Proposition~\ref{prop231}
and Proposition~\ref{prop251}, we have $M=N_{13}$. We will arrange our matrices
similarly to \eqref{eq773}.

We will need the following easy and general observation:

\begin{lemma}\label{lem253}
Let $M$ be any of the $N_m$'s  and $i,j\in\{1,2,3\}$.
If, for some $s$, the column $s$ in the matrix $M_{ij}$ is non-zero,
then the column $s$ is non-zero in $M_{tj}$, for any $t\in\{1,2,3\}$.
\end{lemma}

\begin{proof}
The fact that the column $s$ in $M_{ij}$ is non-zero is equivalent to saying that
$\mathrm{F}_{ij}\, P_s\neq 0$ (and similarly for $\mathrm{F}_{tj}$). We have $\mathrm{F}_{it} \circ
\mathrm{F}_{tj}\cong \mathrm{F}_{ij}$ from \eqref{eq1n}. Therefore $\mathrm{F}_{tj}\, P_s=0$
implies $\mathrm{F}_{ij}\, P_s= 0$ and the claim follows.
\end{proof}

\begin{proposition}\label{prop271}
The only possibility for the $M_{ij}$'s is
\begin{displaymath}
\left(\begin{array}{ccc}1&0&0\\ 0&0&0\\ 0&0&0\end{array}\right),\quad
\left(\begin{array}{ccc}1&1&0\\ 0&0&0\\ 0&0&0\end{array}\right),\quad
\left(\begin{array}{ccc}0&1&1\\ 0&0&0\\ 0&0&0\end{array}\right),
\end{displaymath}
\begin{displaymath}
\left(\begin{array}{ccc}0&0&0\\ 1&0&0\\ 0&0&0\end{array}\right),\quad
\left(\begin{array}{ccc}0&0&0\\ 1&1&0\\ 0&0&0\end{array}\right),\quad
\left(\begin{array}{ccc}0&0&0\\ 0&1&1\\ 0&0&0\end{array}\right),
\end{displaymath}
\begin{displaymath}
\left(\begin{array}{ccc}0&0&0\\ 0&0&0\\ 1&0&0\end{array}\right),\quad
\left(\begin{array}{ccc}0&0&0\\ 0&0&0\\ 1&1&0\end{array}\right),\quad
\left(\begin{array}{ccc}0&0&0\\ 0&0&0\\ 0&1&1\end{array}\right).
\end{displaymath}
\end{proposition}

\begin{proof}
Due to Lemma~\ref{lem273}, we have to consider three cases.

{\bf Case~1.} Suppose that the pairs of idempotent matrices which share
diagonal elements are $\{\mathrm{F}_{12},\mathrm{F}_{22}\}$ and $\{\mathrm{F}_{23},\mathrm{F}_{33}\}$.
This gives the same contradiction as in Case~2 in Subsection~\ref{s12.6}. Therefore this
case cannot occur.

{\bf Case~2.} Suppose that the pairs of idempotent matrices which share
diagonal elements are $\{\mathrm{F}_{11},\mathrm{F}_{12}\}$ and $\{\mathrm{F}_{23},\mathrm{F}_{33}\}$.
This gives the same contradiction as in Case~3 in Subsection~\ref{s12.6}.
Therefore this case cannot occur either.

{\bf Case~3.} Suppose that the pairs of idempotent matrices
which share  diagonal elements are $\{\mathrm{F}_{11},\mathrm{F}_{12}\}$ and
$\{\mathrm{F}_{22},\mathrm{F}_{23}\}$. Then, using permutation action and
all possible zero products which appear in \eqref{eq1n},
we obtain that the $M_{ij}$'s look as follows:
\begin{displaymath}
\left(\begin{array}{ccc}1&0&0\\ 0&0&0\\ 0&0&0\end{array}\right),\quad
\left(\begin{array}{ccc}1&*&0\\ 0&0&0\\ 0&0&0\end{array}\right),\quad
\left(\begin{array}{ccc}0&*&*\\ 0&0&0\\ 0&0&0\end{array}\right),
\end{displaymath}
\begin{displaymath}
\left(\begin{array}{ccc}0&0&0\\ *&0&0\\ *&0&0\end{array}\right),\quad
\left(\begin{array}{ccc}0&0&0\\ *&1&0\\ *&*&0\end{array}\right),\quad
\left(\begin{array}{ccc}0&0&0\\ 0&1&*\\ 0&*&0\end{array}\right),
\end{displaymath}
\begin{displaymath}
\left(\begin{array}{ccc}0&0&0\\ 0&0&0\\ *&0&0\end{array}\right),\quad
\left(\begin{array}{ccc}0&0&0\\ 0&0&0\\ *&*&0\end{array}\right),\quad
\left(\begin{array}{ccc}0&0&0\\ 0&0&0\\ 0&*&1\end{array}\right).
\end{displaymath}
Using that all matrices must be non-zero and add up to $M$ and also
$M_{13}M_{31}=M_{11}$, $M_{21}M_{12}=M_{22}$ and $M_{21}M_{13}=M_{23}$, given by \eqref{eq1n}, we have
\begin{displaymath}
\left(\begin{array}{ccc}1&0&0\\ 0&0&0\\ 0&0&0\end{array}\right),\quad
\left(\begin{array}{ccc}1&1&0\\ 0&0&0\\ 0&0&0\end{array}\right),\quad
\left(\begin{array}{ccc}0&1&1\\ 0&0&0\\ 0&0&0\end{array}\right),
\end{displaymath}
\begin{displaymath}
\left(\begin{array}{ccc}0&0&0\\ 1&0&0\\ *&0&0\end{array}\right),\quad
\left(\begin{array}{ccc}0&0&0\\ 1&1&0\\ *&*&0\end{array}\right),\quad
\left(\begin{array}{ccc}0&0&0\\ 0&1&1\\ 0&*&0\end{array}\right),
\end{displaymath}
\begin{displaymath}
\left(\begin{array}{ccc}0&0&0\\ 0&0&0\\ 1&0&0\end{array}\right),\quad
\left(\begin{array}{ccc}0&0&0\\ 0&0&0\\ *&*&0\end{array}\right),\quad
\left(\begin{array}{ccc}0&0&0\\ 0&0&0\\ 0&*&1\end{array}\right).
\end{displaymath}
Comparing the first and the second columns in $M_{32}$ with those of  $M_{22}$
and also the third column in $M_{33}$ with that of  $M_{23}$ and using Lemma~\ref{lem253}
we get exactly the arrangement in the formulation of our proposition.
This completes the proof.
\end{proof}

\subsection{Connecting to the cell $2$-representation}\label{s19.2}

Now we know that the $M_{ij}$'s have the form as specifies in Proposition~\ref{prop271}.
For $i,j=1,2,3$, we denote by $\mathrm{G}_{ij}$ the corresponding indecomposable
projective endofunctor of $\overline{\mathbf{M}}(\mathtt{i})$.

From the form of $M_{i1}$, where $i=1,2,3$, we see that $\mathrm{F}_{i1}$ acts via
$\mathrm{G}_{i1}$ (up to isomorphism). Moreover, we also have
$[P_i:L_1]=\delta_{i,1}$.

From the form of $M_{12}$, we see that $\mathrm{F}_{12}$ acts via either
$\mathrm{G}_{12}$ or $\mathrm{G}_{11}$ or $\mathrm{G}_{12}\oplus \mathrm{G}_{11}$.
However, we already know that $\mathrm{G}_{11}$ has matrix $M_{11}$. This leaves us
with the only possibilities of $\mathrm{G}_{12}$ or $\mathrm{G}_{12}\oplus \mathrm{G}_{11}$.

Assume that $\mathrm{F}_{12}$ acts via $\mathrm{G}_{12}\oplus \mathrm{G}_{11}$. Then the matrix of
$\mathrm{G}_{12}$ is
\begin{displaymath}
\left(\begin{array}{ccc}0&1&0\\ 0&0&0\\ 0&0&0\end{array}\right).
\end{displaymath}
This implies that $[P_i:L_2]=\delta_{i,2}$, for $i=1,2,3$.

According to \eqref{eq101nn}, there is a non-zero $2$-morphism $\alpha:\mathrm{F}_{21}\to \mathrm{F}_{12}$.
As $\mathbf{M}$ is faithful, $\overline{\mathbf{M}}(\alpha)$ is non-zero. Evaluation of the latter at
\begin{itemize}
\item $P_3$ is zero as $P_3$ is annihilated by both $\mathrm{F}_{21}$ and $\mathrm{F}_{12}$;
\item $P_2$ is zero as $P_2$ is annihilated by $\mathrm{F}_{21}$;
\item $P_1$ is zero as $\mathrm{F}_{12}\, P_1\cong P_1$, $\mathrm{F}_{21}\, P_1\cong P_2$
and we also have
\begin{displaymath}
\mathrm{Hom}_{\overline{\mathbf{M}}(\mathtt{i})}(P_2,P_1)=[P_1:L_2]=0,
\end{displaymath}
by the previous paragraph.
\end{itemize}
Therefore $\overline{\mathbf{M}}(\alpha)$ must be zero, a contradiction.
Consequently, $\mathrm{F}_{12}$ acts via $\mathrm{G}_{12}$, which also implies $[P_1:L_2]=1$.
From this, it follows that $\mathrm{F}_{i2}$ acts via $\mathrm{G}_{i2}$, for $i=1,2,3$.

A similar argument shows that $\mathrm{F}_{i3}$ acts via $\mathrm{G}_{i3}$, for $i=1,2,3$,
and that
\begin{displaymath}
[P_2:L_3]=[P_3:L_3]=1,\quad [P_3:L_1]=[P_3:L_2]=0.
\end{displaymath}
This means that $B\cong A$ and that
each $\mathrm{F}_{ij}$ acts via the corresponding $\mathrm{G}_{ij}$. It now
follows by the usual arguments, see \cite[Proposition~9]{MM5}, that
$\overline{\mathbf{M}}$ is equivalent to a cell $2$-representation of $\cC_A$.
The claim of Theorem~\ref{mainresult} for the algebra
$\Bbbk(\xymatrix{\bullet\ar[r]\ar@/^/@{.}[rr]&\bullet\ar[r]&\bullet})$ follows.


\noindent
V.~M: Department of Mathematics, Uppsala University, Box. 480,
SE-75106, Uppsala, SWEDEN, email: {\tt mazor\symbol{64}math.uu.se}

\noindent
X.~Z: Department of Mathematics, East China Normal University, Minhang District,
Dong Chuan Road 500, Shanghai, 200241, P. R. CHINA,\\
email: {\tt scropure\symbol{64}126.com}

\end{document}